\documentclass[12pt,reqno]{amsart}

\usepackage[cp1251]{inputenc}
\usepackage[T2A]{fontenc}
\usepackage[english]{babel}
\usepackage{amsmath, amsthm, amscd, amsfonts, amssymb, graphicx, color}
\usepackage[bookmarksnumbered, colorlinks, plainpages]{hyperref}
\usepackage{geometry}
\usepackage{cite}
\newtheorem{theorem}{Theorem}[section]

\newtheorem{lemma}[theorem]{Lemma}
\newtheorem{proposition}{Proposition}

\theoremstyle{definition}

\newtheorem{remark}{Remark}

\title[Elliptic operators on refined Sobolev scales on vector bundles]{Elliptic operators on refined Sobolev scales \\on vector bundles}

\author[T. Zinchenko]{Tetiana Zinchenko}

\address{Chernihiv National Pedagogical University named after T. Shevchenko, Ukraine}

\email{zinchenkotat@ukr.net}



\begin{document}

\maketitle

 \begin{abstract}
{We introduce a refined Sobolev scale on a vector bundle over a closed infinitely smooth manifold. This scale consists of inner product H\"ormander spaces parametrized  with a real number and a function varying slowly at infinity in the sense of Karamata.  We prove that these spaces are obtained by the interpolation with a function parameter between inner product Sobolev spaces.
An arbitrary classical elliptic pseudodifferential operator acting between vector bundles of the same rank is investigated on this scale. We prove that this operator is bounded and Fredholm on pairs of  appropriate H\"ormander spaces. We also prove that the solutions to the corresponding  elliptic equation satisfy a certain a priori estimate on these spaces. The local regularity of these solutions is investigated  on the refined Sobolev scale. We find new sufficient conditions for the solutions to have continuous derivatives of a given order.}
\end{abstract}

\section{Introduction}

It is well known \cite{Hermander87, Wells} that elliptic differential and pseudodifferential operators on a closed infinitely smooth manifold are Fredholm between appropriate Sobolev spaces. This fundamental property is used in the theory of elliptic differential equations and elliptic boundary-value problems. However, the Sobolev scale is not sufficiently finely calibrated for some mathematical problems (see monographs \cite{Hermander63, Hermander83, Jacob010205, MazyaShaposhnikova09, MikhailetsMurach14, NicolaRodino10, Paneah00, Triebel01}). In this connection, H\"ormander \cite{Hermander63, Hermander83} introduced and investigated a broad class of normed function spaces
$$
\mathcal{B}_{p,\mu}=\bigl\{w\in\mathcal{S}'(\mathbb{R}^{n}): \mu\widehat{w}\in L_{p}(\mathbb{R}^{n})\bigr\}, \quad \|w\|_{\mathcal{B}_{p,\mu}} := \|\mu\widehat{w}\|_{L_{p}(\mathbb{R}^{n})},
$$
where $1\leq p\leq\infty$, $\mu: \mathbb{R}^{n} \to (0, \infty)$ is a weight function, and  $\widehat{w}$ is the Fourier transform of a tempered distribution $w$. H\"ormander applied these spaces to investigation of solvability of partial differential equations given in Euclidean domains and to study of regularity of solutions to these equations.

Nevertheless, the class of all spaces $\mathcal{B}_{p,\mu}$ is too general for applications to differential equations on manifolds and boundary-value problems. Among these spaces, Mikhailets and Murach \cite{MikhailetsMurach05UMJ5, MikhailetsMurach06UMJ2, MikhailetsMurach06UMJ3} selected the class of inner product spaces $H^{s,\varphi}:= \mathcal{B}_{2,\mu}$ parametrized with the function $\mu(\xi)=\langle\xi\rangle^s\varphi(\langle\xi\rangle)$, where $s\in\mathbb{R}$, the function $\varphi: [1,\infty) \to (0,\infty)$ varies slowly at infinity in the sense of Karamata \cite{Karamata30a, Karamata30b}, and $\langle\xi\rangle = (1+|\xi|^2)^{1/2}$. This class is called the refined Sobolev scale. It consists of inner product Sobolev spaces $H^{s} = H^{s,1}$ and is obtained by the interpolation with a function parameter between these spaces.
This interpolation property allowed Mikhailets and Murach \cite{MikhailetsMurach05UMJ5, MikhailetsMurach06UMJ2, MikhailetsMurach06UMJ3, MikhailetsMurach06UMJ11, MikhailetsMurach07UMJ5, Murach07UMJ6, MikhailetsMurach08UMJ4, Murach08MFAT2, MikhailetsMurach12BJMA2} to build the theory of solvability of general elliptic systems and elliptic boundary--value problems on the refined Sobolev scale. Their theory \cite{MikhailetsMurach14} is supplemented in \cite{Murach09UMJ3, MikhailetsMurach13UMJ3, ZinchenkoMurach12UMJ11, MurachZinchenko13MFAT1, ZinchenkoMurach14JMathSci, AnopMurach14MFAT2, AnopMurach14UMJ, AnopKasirenko16MFAT4} for a more extensive class of H\"ormander inner product spaces. The refined Sobolev scale and other classes of H\"ormander spaces are applied to the spectral theory of elliptic differential operators on manifolds \cite[Section~2.3]{MikhailetsMurach14}, theory of interpolation of normed spaces \cite{MikhailetsMurach13UMJ3, MikhailetsMurach15ResMath1}, to some differential-operator equations \cite{IlkivStrap15}, parabolic initial-boundary value problems \cite{LosMurach13MFAT2, Los15UMG5, LosMikhailetsMurach17CPAA, LosMurach17OpenMath}, in mathematical physics \cite{MikhailetsMolyboga09MFAT1, MikhailetsMolyboga11MFAT3}.

However, elliptic operators on vector bundles have not been covered by this theory. These operators have important applications to elliptic boundary problems on vector bundles \cite{Hermander87}, elliptic complexes, spectral theory of elliptic differential operators, and others (see, e.g., \cite{Hermander87, Wells, Agranovich94}).

The goal of this paper is to introduce and investigate the refined Sobolev scale on an arbitrary vector bundle over infinitely smooth closed manifold and to give applications of this scale to general elliptic pseudodifferential operators on vector bundles.

The paper consists of eight sections. Section 1 is Introduction. In Section 2, we introduce the refined Sobolev scale on the vector bundle. Section 3 is devoted to the method of interpolation with a function parameter between Hilbert spaces. This method plays a key role in the paper. Section 4 contains main results concerning properties of the refined Sobolev scale introduced. Section 5 presents main results about properties of elliptic pseudodifferential operators on this scale.  Section 6 contains some auxiliary facts. The main results of the paper formulated in Sections 4 and 5 are proved in Sections 7 and 8 respectively.

\section{The refined Sobolev scale on a vector bundle}\label{section2}

The refined Sobolev scale on $\mathbb{R}^{n}$ and smooth manifolds was introduced and investigated by Mikhalets and Murach \cite{MikhailetsMurach06UMJ3, MikhailetsMurach08MFAT1}. This scale consists of the inner product H\"ormander spaces $H^{s,\varphi}$ with $s\in\mathbb{R}$ and $\varphi\in\mathcal{M}$. Let us give the definition of the function class  $\mathcal{M}$ and the space $H^{s,\varphi}(\mathbb{R}^{n})$. The latter will be a base for our definition of the refined Sobolev scale on vector bundles.

The class $\mathcal{M}$ consists of all Borel measurable functions $\varphi: [1,\infty)\rightarrow(0,\infty)$  that satisfy the following two conditions:
\begin{itemize}
\item [(i)] both functions $\varphi$ and $1/\varphi$ are bounded on each compact interval $[1,b]$ with $1<b<\infty$;
\item [(ii)] the function $\varphi$ varies slowly at infinity in the sense of Karamata \cite{Karamata30a}, i.e.
$$
\lim_{t\rightarrow\infty}\frac{\varphi(\lambda t)}{\varphi(t)}=1  \quad \mbox{for every}\quad \lambda>0.
$$
\end{itemize}

Slowly varying functions are well investigated and play an important role in mathematical analysis and its applications (see monographs \cite{BinghamGoldieTeugels89, GelukHaan87, Seneta76}). A standard example of a function $\varphi \in \mathcal{M}$ is given by  a continuous function  $\varphi: [1,\infty)\rightarrow(0,\infty)$ such that
\begin{equation}\label{f12}
\varphi(t):=(\log t)^{r_{1}}(\log\log t)^{r_{2}}\ldots(\underbrace{\log\ldots\log}_{k} t)^{r_{k}}\quad
\mbox{for}\quad t\gg1,
\end{equation}
where $0\leq k \in \mathbb{Z}$ and $r_1, \ldots, r_k \in  \mathbb{R}$.

The class $\mathcal{M}$ admits the following description (see, e.g., \cite[Section~1.2]{Seneta76}):
$$
\varphi\in\mathcal{M}\;\Longleftrightarrow\;\varphi(t)=\exp\Biggl(\beta(t)+
\int\limits_{1}^{t}\frac{\alpha(\tau)}{\tau}\,d\tau\Biggr)\;\;
\mbox{for}\;\;t\geq1.
$$
Here, $\alpha$ is a continuous function on $[1,\infty)$ such that $\alpha(\tau)\rightarrow0$ as
$\tau\rightarrow\infty$, and $\beta$ is a Borel measurable function on $[1,\infty)$ such that
$\beta(t)\rightarrow l$ as  $t\rightarrow\infty$ for some $l\in \mathbb{R}$.

Let $s\in\mathbb{R}$ and $\varphi\in\mathcal{M}$. By definition, the complex linear space $H^{s,\varphi}(\mathbb{R}^{n})$, with $1\leq n\in\mathbb{Z}$, consists of all distributions $w\in\mathcal{S}'(\mathbb{R}^{n})$ such that their Fourier transform
$\widehat{w}$ is locally Lebesgue integrable over $\mathbb{R}^{n}$ and satisfies the condition
$$
\int\limits_{\mathbb{R}^{n}}
\langle\xi\rangle^{2s}\varphi^{2}(\langle\xi\rangle)\,
|\widehat{w}(\xi)|^{2}\,
d\xi<\infty.
$$
Here, $\mathcal{S}'(\mathbb{R}^{n})$ is the complex linear topological space of all tempered distributions on $\mathbb{R}^{n}$, and   $\langle\xi\rangle = (1+|\xi|^2)^{1/2}$. An inner product in $H^{s,\varphi}(\mathbb{R}^{n})$ is defined by the formula
$$
(w_{1},w_{2})_{s,\varphi;\mathbb{R}^{n}}:=
\int\limits_{\mathbb{R}^{n}}
\langle\xi\rangle^{2s}\varphi^{2}(\langle\xi\rangle)
\,\widehat{w_{1}}(\xi)\,\overline{\widehat{w_{2}}(\xi)}\,d\xi,
$$
with $w_1, w_2 \in H^{s,\varphi}(\mathbb{R}^{n})$. This inner product endows $H^{s,\varphi}(\mathbb{R}^{n})$ with the Hilbert spaces structure and induces the norm
$$
\|w\|_{s,\varphi;\mathbb{R}^{n}}:= (w,w)_{s,\varphi;\mathbb{R}^{n}}^{1/2}.
$$
The space $H^{s,\varphi}(\mathbb{R}^{n})$ is separable with respect to this norm, and the set $C^\infty_0(\mathbb{R}^{n})$ is dense in this space. Here, as usual, $C^\infty_0(\mathbb{R}^{n})$ stands for the set of all infinitely differentiable compactly supported functions $w: \mathbb{R}^{n} \rightarrow \mathbb{C}$.

In this paper, we consider complex-valued functions and distributions; hence, all function spaces are supposed to be complex. Besides, we interpret distributions as antilinear functionals on corresponding spaces of test functions.

If $\varphi\equiv1$, then the space $H^{s,\varphi}(\mathbb{R}^{n})$ coincides with the inner product Sobolev space $H^{s}(\mathbb{R}^{n})$ of order $s$. Generally, we have the continuous embeddings
\begin{equation}\label{2.34}
H^{s+\varepsilon}(\mathbb{R}^{n})\hookrightarrow
H^{s,\varphi}(\mathbb{R}^{n})\hookrightarrow
H^{s-\varepsilon}(\mathbb{R}^{n})\quad\mbox{for any}\;\;\varepsilon>0.
\end{equation}
They show that the function parameter $\varphi$ defines a supplementary  regularity with respect to the main (power) regularity $s$.  Briefly    saying, $\varphi$ refines main regularity $s$. Following \cite{MikhailetsMurach14, MikhailetsMurach13UMJ3}, we call the class of function spaces
\begin{equation}\label{2.36}
\{H^{s,\varphi}(\mathbb{R}^{n}): s\in\mathbb{R},\varphi\in\mathcal{M}\}
\end{equation}
the refined Sobolev scale on $\mathbb{R}^{n}$.

Let $\Gamma$ be a closed (i. e. compact and without boundary) infinitely smooth real manifold
of dimension $n\geq1$.  We suppose that a certain $C^\infty$-density $dx$ is given on $\Gamma$.
Let $\pi: V \rightarrow\Gamma$ be an infinitely smooth complex vector bundle of rank $p\geq1$ on $\Gamma$. Here,
$V$ is the total space of the bundle, $\Gamma$ is the base  space, and $\pi$ is the projector (see, e.g., \cite[Chapter~I, Section~2]{Wells}). Let $C^\infty(\Gamma, V)$ denote the  complex linear space of all infinitely differentiable sections $u:\Gamma\rightarrow V$. Note that $u(x)\in\pi^{-1}(x)$ for every $x\in \Gamma$ and that $\pi^{-1}(x)$ is a complex vector  space of dimension $p$ (this space is called the fiber over $x$).

Let us introduce the H\"ormander space $H^{s,\varphi}(\Gamma,V)$ on this vector bundle.
From the $C^{\infty}$-structure on $\Gamma$, we choose a finite atlas consisting of local charts $\alpha_{j}:\mathbb{R}^{n}\leftrightarrow\Gamma_{j}$ with $j=1,\ldots,\varkappa$. Here, the open sets  $\Gamma_{j}$ form a finite covering of $\Gamma$. We choose these sets so that the local trivialization
$\beta_{j}:\pi^{-1}(\Gamma_{j})\leftrightarrow
\Gamma_{j}\times\mathbb{C}^{p}$ is defined.
We also choose real-valued functions $\chi_{j}\in C^{\infty}(\Gamma)$, $j=1,\ldots,\varkappa$,
that satisfy the condition $\mathrm{supp}\,\chi_{j}\subset\Gamma_{j}$ and form a partition of unity on~$\Gamma$.

Let $s\in\mathbb{R}$ and $\varphi\in\mathcal{M}$. We introduce the norm on $C^{\infty}(\Gamma, V)$ by the formula
\begin{equation}\label{1.34}
\|u\|_{s,\varphi;\Gamma, V} := \biggl(\,\sum_{j=1}^{\varkappa}\sum_{k=1}^{p} \|(\Pi_k(\beta_j \circ(\chi_j u) \circ\alpha_j)\|^{2}_{s,\varphi;\mathbb{R}^n}\biggr)^{1/2}.
\end{equation}
Here, $u\in C^{\infty}(\Gamma, V)$, and the projector $\Pi_k$ is defined as follows: $\Pi_k: (x,a)\mapsto a_k$ for all $x\in\Gamma$ and $a=(a_1,\ldots,a_p)\in\mathbb{C}^{p}$. We put
\begin{equation}\label{1.134}
u_{j,k}:=\Pi_k(\beta_j \circ(\chi_j u) \circ\alpha_j)
\end{equation}
for  arbitrary $j \in \{ 1, \ldots\varkappa\}$ and $k\in \{1,\ldots, p\}$. Note that if $u\in C^{\infty}(\Gamma, V)$, then each $u_{j,k} \in C^\infty_0(\mathbb{R}^n)$; hence, the norms on the right-hand side of \eqref{1.34} are well defined. The norm  \eqref{1.34} is Hilbert because it is induced by the inner product
\begin{equation}\label{1.35}
(u, v)_{s,\varphi;\Gamma, V} := \sum_{j=1}^{\varkappa}\sum_{k=1}^{p} (u_{j,k}, v_{j,k})_{s,\varphi;\mathbb{R}^n}
\end{equation}
of sections $u, v \in C^{\infty}(\Gamma, V)$.

Let $H^{s,\varphi}(\Gamma, V)$ be the completion of the linear space $C^\infty(\Gamma, V)$  with respect to the norm \eqref{1.34} (and the corresponding inner product \eqref{1.35}). Thus, we have the Hilbert space $H^{s,\varphi}(\Gamma, V)$. This space does not depend up to equivalent of norms on our choice of the atlas $\{\alpha_j\}$, partition of unity $\{\chi_j\}$, and local trivializations $\{\beta_j\}$. This will be proved bellow as Theorem~\ref{t1}. By analogy with \eqref{2.36} we call the class of Hilbert function spaces
\begin{equation}\label{1.36}
\{H^{s,\varphi}(\Gamma, V): s\in\mathbb{R},\varphi\in\mathcal{M}\}
\end{equation}
the refined Sobolev scale on the bundle $\pi:V\to\Gamma$.

If $\varphi=1$, then $H^{s,\varphi}(\Gamma, V)$ becomes the inner product Sobolev space  $H^{s}(\Gamma, V)$ of order $s\in\mathbb{R}$ (see e.g. \cite[Chapter~IV, Section~1]{Wells}). In the Sobolev case of $\varphi = 1$, we will omit the index $\varphi$ in our designations concerning the H\"ormander spaces $H^{s,\varphi}(\cdot)$. Specifically, $\|\cdot\|_{s;\Gamma, V}$ denotes the norm in the Sobolev space $H^{s}(\Gamma, V)$.

In the case of trivial vector bundle of rank $p=1$, the space  $H^{s,\varphi}(\Gamma, V)$ consists of distributions on $\Gamma$ and is denoted by $H^{s,\varphi}(\Gamma)$. The space $H^{s,\varphi}(\Gamma)$ was introduced and investigated by Mikhailets and Murach \cite{MikhailetsMurach06UMJ3, MikhailetsMurach08MFAT1}.

\section{Interpolation with function parameter between Hilbert spaces}

The refined Sobolev scale on the vector bundle $\pi:V\to\Gamma$ possesses an important interpolation property. Namely, every space $H^{s,\varphi}(\Gamma, V)$, with $s\in\mathbb{R}$ and $\varphi\in\mathcal{M}$,
is the result of the interpolation with an appropriate function parameter between
the  Sobolev spaces $H^{s-\varepsilon}(\Gamma, V)$ and $H^{s+\delta}(\Gamma, V)$, where $\varepsilon, \delta > 0$. We will systematically use this property in the paper. Therefore we recall the definition of interpolation with
function parameter between Hilbert spaces and discuss some of its properties. We restrict ourselves to the case of separable complex Hilbert spaces and mainly follow monograph \cite[Section 1.1]{MikhailetsMurach14}. Note that the interpolation with function parameter between normed spaces was introduced by Foia\c{s} and Lions \cite{FoiasLions61}, who separately considered the case of Hilbert spaces.

Let $X:=[X_{0},X_{1}]$ be an ordered pair of separable complex Hilbert spaces $X_{0}$ and $X_{1}$ such
that $X_{1}\subset X_{0}$ with the continuous and dense embedding. This pair is said to be admissible.
For $X$ there exists an isometric isomorphism
$J:X_{1}\leftrightarrow X_{0}$ that $J$ is a self-adjoint positive-definite operator in
$X_{0}$ with the domain $X_{1}$. The operator $J$ is uniquely determined by the pair $X$ and is called a generating operator
for this pair.

Let $\mathcal{B}$ denote the set of all Borel measurable
functions $\psi:(0,\infty)\rightarrow(0,\infty)$ that $\psi$ is bounded on every
compact interval $[a,b]$, with $0<a<b<\infty$, and that $1/\psi$ is bounded on every
set $[r,\infty)$, with $r>0$. Given $\psi\in\mathcal{B}$, consider the operator $\psi(J)$ defined as the Borel function $\psi$ of the self-adjoint operator $J$ with the help of Spectral Theorem.  The operator $\psi(J)$ is (generally) unbounded and positive-definite in $X_{0}$. Let $[X_{0},X_{1}]_{\psi}$ or, simply,
$X_{\psi}$ denote the domain of $\psi(J)$ endowed with the inner product
$(u_{1},u_{2})_{X_{\psi}}:=(\psi(J)u_{1},\psi(J)u_{2})_{X_{0}}$ and the
corresponding norm $\|u\|_{X_{\psi}}=\|\psi(J)u\|_{X_{0}}$. The space $X_{\psi}$ is
Hilbert and separable.

A function $\psi\in\mathcal{B}$ is said to be an interpolation parameter if the
following condition is fulfilled for each admissible pairs $X=[X_{0},X_{1}]$ and
$Y=[Y_{0},Y_{1}]$ of Hilbert spaces and for an arbitrary linear mapping $T$ given on
$X_{0}$: if the restriction of $T$ to $X_{j}$ is a bounded operator
$T:X_{j}\rightarrow Y_{j}$ for each $j\in\{0,1\}$, then the restriction of $T$ to
$X_{\psi}$ is also a bounded operator $T:X_{\psi}\rightarrow Y_{\psi}$.

If $\psi$ is an interpolation parameter, then we say that the Hilbert space
$X_{\psi}$ is obtained by the interpolation with the function parameter
$\psi$ between  $X_{0}$ and $X_{1}$ (or of the pair $X$). In this case,
\begin{equation}\label{3.18}
\mbox{the continuous and dense embeddings}\quad X_{1}\hookrightarrow X_{\psi}\hookrightarrow X_{0}
\end{equation}
hold true.

The function $\psi\in\mathcal{B}$ is an interpolation parameter if and only
if $\psi$ is pseudoconcave in a neighborhood of $+\infty$. The latter property means that there exists a concave
function $\psi_{1}:(b,\infty)\rightarrow(0,\infty)$, with $b\gg1$, that both
functions $\psi/\psi_{1}$ and $\psi_{1}/\psi$ are bounded on $(b,\infty)$. This criterion follows from Peetre's \cite{Peetre66, Peetre68} description of all interpolation functions for the weighted Lebesgue spaces (see \cite[Theorem~1.9]{MikhailetsMurach14}). Specifically, every function $\psi\in\mathcal{B}$ of the form $\psi(t)\equiv t^\theta \psi_0(t)$, where $0<\theta< 1$ and $\psi_0$ varies slowly at infinity, is an interpolation parameter.

Let us formulate the above-mentioned interpolation property of the refined Sobolev scale on~$\mathbb{R}^n$ \cite[Theorem~1.14]{MikhailetsMurach14}.

\begin{proposition}\label{p1}
Let a function $\varphi\in\mathcal{M}$ and numbers $\varepsilon, \delta>0$ be given. Put
\begin{equation}\label{3.30}
\psi(t):=
\begin{cases}
\;t^{{\varepsilon}/{(\varepsilon+\delta)}}\,
\varphi(t^{1/{(\varepsilon+\delta)}})&\text{if}\quad t\geq1, \\
\;\varphi(1)&\text{if}\quad0<t<1.
\end{cases}
\end{equation}
Then $\psi\in\mathcal{B}$ is an interpolation parameter, and
\begin{equation*}
H^{s,\varphi}(\mathbb{R}^n)=
[H^{s-\varepsilon}(\mathbb{R}^n),H^{s+\delta}(\mathbb{R}^n)]_{\psi} \quad \mbox{for every} \quad s\in\mathbb{R}
\end{equation*}
with equality of norms.
\end{proposition}

\section{Properties of the refined Sobolev scale on vector bundle}
Let us formulate the main results of the paper concerning properties of the refined Sobolev scale \eqref{1.36} on the vector bundle $\pi:V\to\Gamma$.

\begin{theorem}\label{t6}
Let $\varphi\in\mathcal{M}$ and $\varepsilon, \delta>0$. Define the interpolation parameter $\psi$ by formula \eqref{3.30}.  Then \begin{equation}\label{3.1}
[H^{s-\varepsilon}(\Gamma, V),H^{s+\delta}(\Gamma, V)]_{\psi}=H^{s,\varphi}(\Gamma, V) \quad \mbox{for every} \quad s\in\mathbb{R}
\end{equation}
with equivalence of norms.
\end{theorem}

\begin{theorem}\label{t1}
Let $s\in\mathbb{R}$ and $\varphi\in\mathcal{M}$. The Hilbert space $H^{s,\varphi}(\Gamma, V)$  does not depend up to equivalence of norms on the choice of the atlas $\{\alpha_j\}$ and partition of unity $\{\chi_j\}$ on $\Gamma$ and on the choice of the local trivializations $\beta_j$ of~$V$.
\end{theorem}

\begin{theorem}\label{t8}
Let $s\in\mathbb{R}$ and $\varphi, \varphi_1 \in\mathcal{M}$. The identity mapping $u\mapsto u$, with $u\in C^\infty(\Gamma, V)$, extends uniquely (by continuity) to a compact embedding $H^{s+\varepsilon, \varphi_1}(\Gamma, V)\hookrightarrow H^{s, \varphi}(\Gamma, V)$ for every $\varepsilon>0$.
\end{theorem}

Suppose now that the vector bundle $\pi: V \rightarrow\Gamma$ is Hermitian. Thus, for every $x\in\Gamma$, a certain inner product $\langle \cdot, \cdot \rangle_{x}$ is defined in the fiber $\pi^{-1}(x)$ so that the scalar function $\Gamma\ni x\mapsto\langle u(x), v(x)\rangle_{x}$ is infinitely smooth on $\Gamma$ for arbitrary sections $u,v\in C^\infty(\Gamma, V)$. Using the $C^\infty$-density $dx$ on $\Gamma$, we define the inner product of these sections by the formula
\begin{equation}\label{4.332}
\langle u,v\rangle_{\Gamma, V} := \int\limits_{\Gamma} \langle u(x),v(x)\rangle_{x}\,dx.
\end{equation}

\begin{theorem}\label{t9}
Let $s\in\mathbb{R}$ and $\varphi\in\mathcal{M}$. There exists a number $c=c(s,\varphi)>0$ such that for arbitrary sections $u,v\in C^\infty(\Gamma, V)$ we have the estimate
\begin{equation}\label{4.333}
|\langle u,v\rangle_{\Gamma, V}|\leq c\, \|u\|_{s,\varphi;\Gamma, V}\,\|v\|_{-s,1/\varphi;\Gamma, V}.
\end{equation}
Thus, the form \eqref{4.332}, with $u,v\in C^\infty(\Gamma, V)$, extends by continuity to a sesquilinear form $\langle u,v\rangle_{\Gamma, V}$ defined for arbitrary $u\in H^{s,\varphi}(\Gamma, V)$ and $v\in H^{-s,1/\varphi}(\Gamma, V)$. Moreover,
the spaces $H^{s,\varphi}(\Gamma, V)$ and $H^{-s,1/\varphi}(\Gamma, V)$ are mutually dual (up to equivalence of norms)  with respect to the latter form.
\end{theorem}

In view of this theorem note that $\varphi\in\mathcal{M}\Leftrightarrow 1/\varphi\in\mathcal{M}$; hence, the space $H^{-s,1/\varphi}(\Gamma, V)$ is well defined.

\begin{theorem}\label{t10}
Let $0\leq q\in \mathbb{Z}$ and $\varphi\in\mathcal{M}$. Then the condition
\begin{equation}\label{4.15}
\int\limits_1^\infty \frac{dt}{t\varphi^2(t)}<\infty
\end{equation}
is equivalent to that the identity mapping $u\mapsto u$, with $u\in C^\infty(\Gamma, V)$, extends uniquely to a continuous  embedding $H^{q+n/2,\varphi}(\Gamma, V) \hookrightarrow C^q(\Gamma, V)$. Moreover, this embedding is compact.
\end{theorem}

Here, of course,  $C^q(\Gamma, V)$ denotes the Banach space of all $q$ times continuously differentiable sections $u:\Gamma\rightarrow V$. The norm in this space is defined by the formula
\begin{equation}\label{5.25}
\|u\|_{(q);\Gamma,  V} := \sum_{j=1}^\varkappa\sum_{k=1}^p\|u_{j,k}\|_{(q);\mathbb{R}^n},
\end{equation}
where each $u_{j,k}\in C^q_\mathrm{b}(\mathbb{R}^n)$ is given by \eqref{1.134}. Here, $C^{q}_{\mathrm{b}}(\mathbb{R}^n)$ denotes the Banach space of all $q$ times continuously differentiable functions on $\mathbb{R}^n$ whose partial derivatives up to the $q$-th order are bounded on $\mathbb{R}^n$. This space is endowed with the norm
\begin{equation*}
\|w\|_{(q);\mathbb{R}^n} :=
\sum_{\mu_1 + \cdots + \mu_n\leq q} \sup_{t\in\mathbb{R}^n} \biggl|\frac{\partial^{\mu_1 + \ldots + \mu_n} w(t)}{\partial t_1^{\mu_1}, \ldots, \partial t_n^{\mu_n}}\biggr|
\end{equation*}
of a function $w$.

We will prove Theorems \ref{t6}--\ref{t10} in Section~\ref{sec7}. In the case where $\pi: V \rightarrow \Gamma$ is a trivial vector bundle of rank $p=1$, they are established by Mikhailets and Murach \cite{MikhailetsMurach08MFAT1} (see also their monograph \cite[Section~2.1.2]{MikhailetsMurach14}).

\section{Elliptic operators on the refined Sobolev scale on a vector bundle}

Consider elliptic PsDOs on a pair of vector bundles on $\Gamma$. Let $\pi_1: V_1 \rightarrow\Gamma$ and $\pi_2: V_2 \rightarrow\Gamma$ be two infinitely smooth complex vector bundles of the same rank $p\geq1$ on $\Gamma$.
We choose  the atlas $\{\alpha_j: \mathbb{R}^n \to \Gamma_j\}$ so that both the local trivializations
$\beta_{1,j}:\pi_1^{-1}(\Gamma_{j})\leftrightarrow\Gamma_{j}\times\mathbb{C}^{p}$ and $\beta_{2,j}:\pi_2^{-1}(\Gamma_{j})\leftrightarrow\Gamma_{j}\times\mathbb{C}^{p}$ are defined. We suppose that each vector bundle $\pi_k: V_k \rightarrow\Gamma$ with $k\in\{1,2\}$ is Hermitian. Let $\langle u,v\rangle_{\Gamma, V_k}$ denote the corresponding inner product of sections $u,v\in C^\infty(\Gamma, V_k)$ and its extension by continuity indicated in Theorem~\ref{t9}.

Given $m\in \mathbb{R}$, we let $\Psi^m_{\mathrm{ph}}(\Gamma; V_1, V_2)$ denote the class of all polyhomogeneous (classical) PsDOs $A: C^\infty(\Gamma, V_1) \rightarrow C^\infty(\Gamma, V_2)$ of order $m$ (see, e.g., \cite[Chapter~IV, Section~3]{Wells}). Recall that if $\Omega$ is an open nonempty subset of $\Gamma$ with $\overline{\Omega}\subset\Gamma_j$ for some $j\in \{1,\ldots, \varkappa\}$, then a PsDO $A\in \Psi^m_{\mathrm{ph}}(\Gamma; V_1, V_2)$ is represented locally in the form
\begin{equation}\label{4.26}
\Pi \bigl(\beta_{2,j} \circ(\varphi A(\psi u))\circ \alpha_j\bigr) =
(\varphi\circ\alpha_j)A_{\Omega} \Pi \bigl(\beta_{1,j} \circ(\psi u)\circ \alpha_j \bigr)
\end{equation}
for every section $u\in C^\infty(\Gamma, V_1)$  and  arbitrary scalar functions $\varphi, \psi\in C^\infty(\Gamma)$
with $\mathrm{supp}\,\varphi\subset\Omega$ and $\mathrm{supp}\,\psi\subset\Omega$. Here, $A_{\Omega}$ is a certain $p\times p$-matrix whose entries are polyhomogeneous  PsDOs on $\mathbb{R}^n$ of order $m$, and $\Pi$ is the projector defined by the formula $\Pi: (x,a)\mapsto a$ for arbitrary $x\in\Gamma$ and $a\in \mathbb{C}^p$.
For $A\in \Psi^m_{\mathrm{ph}}(\Gamma; V_1, V_2)$ there is a unique formally adjoint PsDO $A^+\in \Psi^m_{\mathrm{ph}}(\Gamma; V_2, V_1)$ defined by the formula $\langle Au, w\rangle_{\Gamma, V_2} = \langle u, A^+w\rangle_{\Gamma, V_1}$ for all $u\in C^\infty(\Gamma, V_1)$ and $w\in C^\infty(\Gamma, V_2)$.

Hereafter we let $m\in \mathbb{R}$ and suppose that $A$ is an arbitrary elliptic PsDO from the class $\Psi^m_{\mathrm{ph}}(\Gamma; V_1, V_2)$. The ellipticity of $A$ is equivalent to that each operator $A_{\Omega} = (A^{l,r}_{\Omega})^p_{l,r=1}$ from \eqref{4.26} is elliptic on the set $\alpha^{-1}_j(\Omega)$, i.e.
$\det(a^{l,r}_{\Omega,0}(x, \xi))^p_{l,r=1}\neq0$ for all $x\in\alpha^{-1}_j(\Omega)$ and $\xi\in\mathbb{R}^n\setminus\{0\}$, with $a^{l,r}_{\Omega,0}(x, \xi)$ being the principal symbol of the scalar PsDO $A^{l,r}_{\Omega}$ on $\mathbb{R}^n$. Put
\begin{gather}
\mathfrak{N}:=\{ u\in C^\infty(\Gamma, V_1): Au=0 \,\,\mbox{on}\,\, \Gamma\},
\\
\mathfrak{N^+}:=\{ w\in C^\infty(\Gamma, V_2): A^+w=0 \,\,\mbox{on}\,\, \Gamma\}.
\end{gather}
Since the PsDOs $A$ and $A^+$ are elliptic, then the spaces $\mathfrak{N}$ and $\mathfrak{N}^+$ are finite-dimensional (see, e.g., \cite[Theorem~4.8]{Wells}).

\begin{theorem}\label{t2} 
Let $s\in\mathbb{R}$ and $\varphi\in\mathcal{M}$. The mapping $u \mapsto Au$, $u\in C^{\infty}(\Gamma, V_1)$, extends uniquely (by continuity) to a bounded linear operator
\begin{equation}\label{4.222}
A: H^{s+m,\varphi}(\Gamma, V_1) \rightarrow H^{s,\varphi}(\Gamma, V_2).
\end{equation}
This operator is Fredholm. Its kernel is
$\mathfrak{N}$, and its domain
\begin{equation}
A(H^{s+m,\varphi}(\Gamma, V_1)) = \{f\in H^{s,\varphi}(\Gamma, V_2): \langle f,w\rangle_{\Gamma, V_2}=0 \;\mbox{for all}\;  w\in \mathfrak{N^+}\}.
\end{equation}
The index of operator \eqref{4.222} is equal to $\dim\mathfrak{N}-\dim\mathfrak{N}^+$ and does not dependent on $s$ and~$\varphi$.
\end{theorem}

Recall that a bounded linear operator $T: E_1 \rightarrow E_2$ between Banach spaces $E_1$ and $E_2$ is called Fredholm if its kernel $\mathrm{ker}\,T$ and co-kernel $\mathrm{coker}\,T := E_2/T(X)$ are finite-dimensional. If the operator $T$ is Fredholm, then its domain $T(X)$ is closed in $E_2$ and its index $\mathrm{ind}\;T := \dim \ker T - \dim \mathrm{coker}\,T$ is finite (see, e.g. \cite[Lemma~19.1.1]{Hermander87}).

If $\mathfrak{N}=\{0\}$ and $\mathfrak{N^+}=\{0\}$, then operator~\eqref{4.222} is an isomorphism between the spaces $H^{s+m,\varphi}(\Gamma, V_1)$ and $H^{s,\varphi}(\Gamma, V_2)$ by virtue of the Banach theorem of inverse operator. In the general situation, this operator induces an isomorphism between their certain subspaces of finite codimension. In this connection  consider the following decompositions of these spaces into direct sums of their subspaces:
\begin{gather}\label{4.220}
H^{s+m,\varphi}(\Gamma, V_1)=\mathfrak{N}\dotplus\{u\in H^{s+m,\varphi}(\Gamma, V_1):\langle u,v\rangle_{\Gamma, V_1}=0\;\,\mbox{for all}\;v\in \mathfrak{N}\},\\
H^{s,\varphi}(\Gamma, V_2)=\mathfrak{N}^{+}\dotplus\{f\in H^{s,\varphi}(\Gamma, V_2):\langle f,w\rangle_{\Gamma, V_2}=0\;\,\mbox{for all}\;w\in \mathfrak{N}^{+}\}.\label{4.221}
\end{gather}
These decompositions are well defined because the summands in them has the trivial intersection and the finite dimension of the first summand is equal to the codimension of the second one. This equality is due to the following fact: if we consider $\mathfrak{N}$ as a subspace of $H^{-s-m,1/\varphi}(\Gamma, V_1)$, then the dual of $\mathfrak{N}$ with respect to the form $\langle \cdot, \cdot\rangle_{\Gamma, V_1}$  coincides with the second summand in \eqref{4.220} according to Theorem~\ref{t9},  analogous reasoning being valid for \eqref{4.221}.

Let $\textit{P}$ and $\textit{P}^+$ denote the oblique projectors of the spaces $H^{s+m,\varphi}(\Gamma, V_1)$ and $H^{s,\varphi}(\Gamma, V_2)$ onto the second summands parallel to  the first summands in \eqref{4.220} and \eqref{4.221} respectively. These projectors are independent of $s$ and $\varphi$.

\begin{theorem}\label{t7}
Let $s\in\mathbb{R}$ and $\varphi\in\mathcal{M}$. The restriction of operator \eqref{4.222} to the subspace $P(H^{s+m,\varphi}(\Gamma, V_1))$ is an isomorphism
\begin{equation}\label{4.235}
A: P(H^{s+m,\varphi}(\Gamma, V_1)) \leftrightarrow P^+(H^{s,\varphi}(\Gamma, V_2)).
\end{equation}
\end{theorem}

The solutions $u\in H^{s+m,\varphi}(\Gamma, V_1)$ to the elliptic equation $Au=f$ satisfy the following a priori estimate.

\begin{theorem}\label{t3}
Let $s\in\mathbb{R}$, $\varphi\in\mathcal{M}$, and $\sigma<s+m$. Besides, let functions $\chi,\eta\in C^\infty(\Gamma)$ be chosen so that $\eta=1$ in a neighbourhood of $\mathrm{supp}\chi$. Then there exists a number $c>0$ that, for any sections $u\in H^{s+m,\varphi}(\Gamma, V_1)$ and $f\in H^{s,\varphi}(\Gamma, V_2)$ satisfying the equation $Au=f$ on~$\Gamma$, we have the estimate
\begin{equation}\label{f4}
\|\chi u\|_{s+m,\varphi;\Gamma, V_1}\leq c\,\bigl(\|\eta f\|_{s,\varphi;\Gamma, V_2}+\|u\|_{\sigma;\Gamma, V_1}\bigr).
\end{equation}
\end{theorem}

\begin{remark}\label{r1}
If $\chi=1$ on $\Gamma$, then \eqref{f4} becomes  the global estimate
\begin{equation}\label{5.1}
\| u\|_{s+m,\varphi;\Gamma, V_1}\leq c\,\bigl(\| f\|_{s,\varphi;\Gamma, V_2}+\|u\|_{\sigma;\Gamma, V_1}\bigr).
\end{equation}
If $s+m-1<\sigma<s+m$, then we can take $\eta:=\chi$ in  \eqref{f4}; this follows in view of Theorem~\ref{t8} from the estimate
\begin{equation}\label{5.2}
\|\chi u\|_{s+m,\varphi;\Gamma, V_1}\leq c\,\bigl(\|\chi f\|_{s,\varphi;\Gamma, V_2}+\|u\|_{s+m-1,\varphi;\Gamma, V_1}\bigr).
\end{equation}
\end{remark}

Consider the local regularity of the solutions to the elliptic equation $Au=f$.
Given $j\in\{1,2\}$, we put
$$
H^{-\infty}(\Gamma, V_j):= \bigcup_{\sigma\in\mathbb{R}} H^{\sigma}(\Gamma, V_j)=\bigcup_{\sigma\in\mathbb{R}, \eta\in\mathcal{M}} H^{\sigma,\eta}(\Gamma, V_j)
$$
(the latter equality  is due to Theorem~\ref{t8}). Assume that $\Gamma_0$ is an arbitrary open nonempty subset of $\Gamma$. We let  $H_{\mathrm{loc}}^{s,\varphi}(\Gamma_0,V_j)$, with  $s\in\mathbb{R}$ and $\varphi\in\mathcal{M}$, denote the linear space of all sections $v\in H^{-\infty}(\Gamma, V_j)$ such that $\chi v\in H^{s,\varphi}(\Gamma,V_j)$ for every function $\chi \in C^\infty(\Gamma)$ with $\mathrm{supp}\,\chi\subset \Gamma_0$. Here, the product $\chi v\in H^{-\infty}(\Gamma, V_j)$  is well defined by closer.

\begin{theorem}\label{t5}
Let $u\in H^{-\infty}(\Gamma, V_1)$ be a solution to the elliptic equation  $Au=f$ on $\Gamma$ for a certain section
 $f\in H_{\mathrm{loc}}^{s,\varphi}(\Gamma_0, V_2)$. Then $u\in H_{\mathrm{loc}}^{s+m,\varphi}(\Gamma_0, V_1)$.
\end{theorem}

As we can see, the supplementary regularity $\varphi$ is inherited by the solution. If $\Gamma_0=\Gamma$, then the local spaces  $H_{\mathrm{loc}}^{s+m,\varphi}(\Gamma_0, V_1)$ and $H_{\mathrm{loc}}^{s,\varphi}(\Gamma_0, V_2)$ becomes $H^{s+m,\varphi}(\Gamma, V_1)$ and $H^{s,\varphi}(\Gamma, V_2)$ respectively and then Theorem~\ref{t5} says about the global regularity on $\Gamma$.

As an application of the refine Sobolev scale, we give the following result:

\begin{theorem}\label{t4}
Let $0\leq q\in \mathbb{Z}$. Suppose that a section $u\in H^{-\infty}(\Gamma, V_1)$ is a solution to the elliptic equation $Au=f$ on $\Gamma$ where $f\in H_{\mathrm{loc}}^{q-m+n/2,\varphi}(\Gamma_0, V_2)$  for a certain function $\varphi\in\mathcal{M}$ subject to  condition~\eqref{4.15}. Then $u\in C_{\mathrm{loc}}^q(\Gamma_0, V_1)$.
\end{theorem}

Here, $C_{\mathrm{loc}}^q(\Gamma_0, V_1)$ denotes the linear space of all section $u\in H^{-\infty}(\Gamma, V_1)$ such that $\chi u\in C^q(\Gamma, V_1)$ for arbitrary $\chi\in C^\infty(\Gamma)$ with $\mathrm{supp}\, \chi\subset\Gamma_0$.

\begin{remark}\label{r2}
Let $ \varphi\in \mathcal{M}$. Condition~\eqref{4.15} is sharp in Theorem~\ref{t4}. Namely, this condition is equivalent to the implication
\begin{equation}\label{4.16}
\bigl(u\in H^{-\infty}(\Gamma, V_1), \;Au\in H_{\mathrm{loc}}^{q-m+n/2,\varphi}(\Gamma_0, V_2)\bigr) \Rightarrow u\in C_{\mathrm{loc}}^q(\Gamma_0, V_1).
\end{equation}
\end{remark}

We will prove Theorems \ref{t2}--\ref{t4}, formula~\eqref{5.2} in Remark~\ref{r1}, and Remark~\ref{r2} in Section~\ref{sec8}.  In the case where both $\pi_1: V_1 \rightarrow \Gamma$ and $\pi_2: V_2 \rightarrow \Gamma$ are trivial vector bundles of rank $p=1$, these theorems are proved by Mikhailets and Murach \cite{MikhailetsMurach08MFAT1} (see also their monograph \cite[Sections 2.2.2 and 2.2.3]{MikhailetsMurach14}).

\section{Auxiliary results}

We will use three properties of the interpolation with a function parameter.  The first of them reduces the interpolation between orthogonal sums of Hilbert spaces to the interpolation between the summands (see, e.g., \cite[Theorem~1.5]{MikhailetsMurach14}).

\begin{proposition}\label{p2}
Let $\bigl[X_{0}^{(j)},X_{1}^{(j)}\bigr]$, with $j=1,\ldots,r$, be a finite
collection of admissible couples of Hilbert spaces. Then for every function
$\psi\in\mathcal{B}$ we have
$$
\biggl[\,\bigoplus_{j=1}^{r}X_{0}^{(j)},\,\bigoplus_{j=1}^{r}X_{1}^{(j)}\biggr]_{\psi}=\,
\bigoplus_{j=1}^{r}\bigl[X_{0}^{(j)},\,X_{1}^{(j)}\bigr]_{\psi}
$$
with equality of norms.
\end{proposition}

The second property shows that this interpolation preserves the Fredholm property of the bounded operators that have the same defect (see, e.g., \cite[Theorem~1.7]{MikhailetsMurach14}).

\begin{proposition}\label{p4}
Let  $X=[X_{0},X_{1}]$ and $Y=[Y_{0},Y_{1}]$ be admissible pairs of Hilbert spaces, and let $\psi\in\mathcal{B}$ be an interpolation parameter. Suppose that a linear mapping $T$ is given on $X_{0}$ and satisfies the following property: the restrictions of $T$ to the spaces $X_j$, where $j=0,\,1$, are Fredholm bounded operators $T:X_{j}\rightarrow Y_{j}$ that have a common kernel and  the same index. Then the restriction of $T$ to the space $X_{\psi}$ is a Fredholm bounded operator $T:X_{\psi}\rightarrow Y_{\psi}$ with the same kernel and index and, besides, $T(X_{\psi}) = Y_{\psi}\cap T(X_{\,0})$.
\end{proposition}

The third property reduces the interpolation between the dual or antidual spaces of given Hilbert spaces to the interpolation between these given spaces (see \cite[Theorem~1.4]{MikhailetsMurach14}). We need this property in the case of antidual spaces. If $H$ is a Hilbert space, then $H'$ stands for the antidual of $H$; namely, $H'$ consists of all antilinear continuous functionals $l: H \rightarrow \mathbb{C}$. The linear space $H'$ is Hilbert with respect to the inner product  $(l_1, l_2)_{H'} := (v_1, v_2)_{H}$ of functionals $l_1, l_2\in H'$; here $v_j$, with $j\in \{1, 2\}$, is a unique vector from $H$ such that $l_j(w) = (v_j, w)_H$ for every $w\in H$. Note that we do not identify $H$ and $H'$ on the base of the Riesz theorem (according to which  $v_j$ exists).

\begin{proposition}\label{p5}
Let a function $\psi\in\mathcal{B}$ be such that the function $\psi(t)/t$ is bounded in a neighbourhood  of infinity. Than for every admissible pair $[X_0, X_1]$ of Hilbert spaces we have the equality  $[X_1', X_0']_\psi = [X_0, X_1]'_\chi$ with equality of norms. Here, the function $\chi\in\mathcal{B}$ is defined by the formula $\chi(t):=t/\psi(t)$ for $t>0$. If $\psi$ is an interpolation parameter, then $\chi$ is an interpolation parameter as well.
\end{proposition}

In view of this theorem we note that if $[X_0, X_1]$ is an admissible pair of Hilbert spaces, then the dual pair $[X_1', X_0']$ is also admissible provided that we identify functions from $X_0'$ with their restrictions on $X_1$.

\section{Proofs of properties of the refined Sobolev scale}\label{sec7}

In this section we will prove Theorems  \ref{t6}--\ref{t10}.

\begin{proof}[Proof of Theorem $\ref{t6}$]
Let $s\in \mathbb{R}$.
As is known \cite[p.~110]{Wells}, the pair of Sobolev spaces on the left of equality \eqref{3.1} is admissible.
We will deduce this equality from Proposition~\ref{p1} with the help of some operators of flattening  and sewing
of the vector bundle $\pi: V \rightarrow\Gamma$.

We define the flattening operator by the formula
\begin{equation}\label{7.2}
T: u\mapsto (u_{1,1}, \ldots, u_{1,p}, \ldots, u_{\varkappa,1}, \ldots, u_{\varkappa,p}) \quad \mbox{for arbitrary} \quad u\in C^\infty(\Gamma, V).
\end{equation}
Here, each function $u_{j, k} \in C_0^\infty(\mathbb{R}^n)$ is defined by formula \eqref{1.134}.
According to \eqref{1.34}, the norm of $u$ in the space $H^{s,\varphi}(\Gamma, V)$ is equal to the norm of $Tu$ in the Hilbert space $(H^{s,\varphi}(\mathbb{R}^n))^{p\varkappa}$. Therefore the linear mapping $u\mapsto Tu$, with $u\in C^\infty(\Gamma, V)$, extends continuously to the isometric operator
\begin{equation}\label{3.3}
T: H^{s,\varphi}(\Gamma, V)\rightarrow (H^{s,\varphi}(\mathbb{R}^n))^{p\varkappa}.
\end{equation}
Besides, this mapping extends continuously to the isometric operators
\begin{equation}\label{3.2}
T: H^\sigma(\Gamma, V)\rightarrow (H^\sigma(\mathbb{R}^n))^{p\varkappa}, \quad \mbox{with} \quad \sigma \in \mathbb{R},
\end{equation}
between Sobolev spaces. Since $\psi$ is an interpolation parameter, it follows from the boundedness of the linear operators \eqref{3.2} with $\sigma \in \{s-\varepsilon, s+\delta\}$ that the restriction of the operator \eqref{3.2} with $\sigma = s-\varepsilon$ is a bounded operator
 \begin{equation}\label{3.4}
T: \bigl[H^{s-\varepsilon}(\Gamma, V),H^{s+\delta}(\Gamma, V)\bigr]_{\psi}\rightarrow \bigl[(H^{s-\varepsilon}(\mathbb{R}^n))^{p\varkappa}, (H^{s+\delta}(\mathbb{R}^n))^{p\varkappa}\bigr]_{\psi}.
\end{equation}
Owing to Propositions~\ref{p1} and \ref{p2}, the target space of \eqref{3.4} takes the form
\begin{equation}\label{3.5}
\begin{aligned}
\bigl[(H^{s-\varepsilon}(\mathbb{R}^n))^{p\varkappa}, (H^{s+\delta}(\mathbb{R}^n))^{p\varkappa}\bigr]_{\psi} & = \bigl([H^{s-\varepsilon}(\mathbb{R}^n),H^{s+\delta}(\mathbb{R}^n)]_{\psi}\bigr)^{p\varkappa} \\
&= \bigl(H^{s,\varphi}(\mathbb{R}^n)\bigr)^{p\varkappa}.
\end{aligned}
\end{equation}
Thus, \eqref{3.4} is a bounded operator between the spaces
\begin{equation}\label{3.6}
T: [H^{s-\varepsilon}(\Gamma, V),H^{s+\delta}(\Gamma, V)]_{\psi}\rightarrow \bigl(H^{s,\varphi}(\mathbb{R}^n)\bigr)^{p\varkappa}.
\end{equation}

Consider now the mapping of sewing
\begin{equation}\label{3.7}
K: (w_{1,1}, \ldots, w_{1,p}, \ldots, w_{\varkappa,1}, \ldots, w_{\varkappa,p}) \mapsto \sum_{j=1}^\varkappa w_j
\end{equation}
defined on vectors
\begin{equation}\label{3.7a}
\mathbf{w} := (w_{1,1}, \ldots, w_{1,p}, \ldots, w_{\varkappa,1}, \ldots, w_{\varkappa,p}) \in (C^\infty_0(\mathbb{R}^n))^{p\varkappa}.
\end{equation}
Here, for each $j\in\{1, \ldots, \varkappa\}$,  the section $w_j\in C^\infty(\Gamma, V)$ is defined by the formula
\begin{equation}\label{3.7b}
w_j (x) := \left\{
        \begin{array}{ll}
         \beta_j^{-1}\bigl(x, (\eta_j w_{j,1})(\alpha_j^{-1}(x)), \ldots, (\eta_j w_{j,p})(\alpha_j^{-1}(x))\bigr) & \hbox{if\;\; $x\in \Gamma_j$,} \\
          0 & \hbox{if\;\; $x\in \Gamma\setminus\Gamma_j$,}
        \end{array}
      \right.
\end{equation}
in which the function $\eta_{j}\in C_0^{\infty}(\mathbb{R}^{n})$ is chosen so that  $\eta_{j} = 1$ on the set
$\alpha_{j}^{-1}(\mathrm{supp}\,\chi_{j})$.

We have the linear mapping
\begin{equation}\label{3.10}
K: (C^\infty_0(\mathbb{R}^n))^{p\varkappa} \to C^\infty(\Gamma, V).
\end{equation}
It is left inverse to the flattening mapping \eqref{7.2}. Indeed, given $u\in C^{\infty}(\Gamma, V)$, we write
\begin{equation*}
KTu = K(u_{1,1}, \ldots, u_{1,p}, \ldots, u_{\varkappa,1}, \ldots, u_{\varkappa,p}) = \sum_{j=1}^\varkappa u_j,
\end{equation*}
where each section $u_j\in C^\infty(\Gamma, V)$ is defined by formula \eqref{3.7b} with $u$ instead of $w$. In this formula, for arbitrary $k\in\{1, \ldots, p\}$ and $x\in \Gamma_j$, we have the equalities
\begin{align*}
(\eta_j u_{j,k})(\alpha_j^{-1}(x)) = \bigl(\eta_j(\alpha_j^{-1}(x))\bigr)\,\Pi_k\bigl(\beta_j ((\chi_j u)(x))\bigr) = \Pi_k\bigl(\beta_j ((\chi_j u)(x))\bigr)
\end{align*}
due to our choice of $\eta_j$. Therefore
\begin{align*}
u_j (x) & = \beta_j^{-1}\bigl(x,(\eta_j u_{j,1})(\alpha_j^{-1}(x)), \ldots, (\eta_j u_{j,p})(\alpha_j^{-1}(x))\bigr) \\
& = \beta_j^{-1}\bigl(x,\Pi_1\bigl(\beta_j ((\chi_j u)(x))\bigr), \ldots, \Pi_p\bigl(\beta_j ((\chi_j u)(x))\bigr)\bigr) \\
& = (\chi_j u)(x)
\end{align*}
for every $x\in\Gamma_j$. Note that if $x\in\Gamma\setminus\Gamma_j$, then $u_j (x) = 0 = (\chi_j u)(x)$. Thus,  $u_j (x) =  (\chi_j u)(x)$ for arbitrary $x\in\Gamma$. Hence,
\begin{equation}\label{3.8}
KTu = \sum_{j=1}^\varkappa u_j = \sum_{j=1}^\varkappa \chi_j u = u \quad \mbox{for every} \quad u\in C^{\infty}(\Gamma, V).
\end{equation}

Let us prove that the mapping \eqref{3.10} extends uniquely to a linear bounded operator between the spaces $(H^{s,\varphi}(\mathbb{R}^n))^{p\varkappa}$ and $H^{s,\varphi}(\Gamma, V)$. Given a vector \eqref{3.7a}, we write
\begin{equation}\label{3.8b}
\begin{aligned}
\| K\mathbf{w} \|_{s,\varphi;\Gamma, V}^2 &=
\sum_{l=1}^\varkappa\sum_{k=1}^p \|\Pi_k\bigl(\beta_l\circ(\chi_l K\mathbf{w})\circ\alpha_l\bigr)\|_{s,\varphi;\mathbb{R}^n}^2
\\
&
= \sum_{l=1}^\varkappa\sum_{k=1}^p \,\biggl\|\sum_{j=1}^\varkappa \Pi_k\bigl(\beta_l\circ(\chi_l w_j)\circ\alpha_l\bigr)\biggr\|_{s,\varphi;\mathbb{R}^n}^2.
\end{aligned}
\end{equation}

Examine the function  $(\chi_l w_j)\circ\alpha_l: \mathbb{R}^n \mapsto \pi^{-1}(\Gamma_l)$  with
$l,j \in \{1, \ldots, \varkappa\}$. If $t\in \mathbb{R}^n$ satisfies $\alpha_l(t)\in\Gamma_j$, then
\begin{align*}
((\chi_l w_j)\circ\alpha_l)(t) = & (\chi_l \circ\alpha_l)(t)\!\cdot\!(w_j\circ\alpha_l)(t) \\
= &
(\chi_l \circ\alpha_l)(t) \!\cdot\! \beta_j^{-1}\bigl(\alpha_l(t), (\eta_j w_{j,1})((\alpha_j^{-1}\circ\alpha_l)(t)), \ldots, (\eta_j w_{j,p})((\alpha_j^{-1}\circ\alpha_l)(t))\bigr)
\\
= &
\beta_j^{-1}\bigl(\alpha_l(t),(\chi_l \circ\alpha_l)(t) \!\cdot\! (\eta_j w_{j,1})((\alpha_j^{-1}\circ\alpha_l)(t)), \ldots,
\\
&\qquad\qquad\;\;
 (\chi_l \circ\alpha_l)(t) \!\cdot\!(\eta_j w_{j,p})((\alpha_j^{-1}\circ\alpha_l)(t))\bigr)
 \\
= &
\beta_j^{-1}\bigl(\alpha_l(t),((\eta_{j,l} w_{j,1})\circ\alpha_{j,l})(t), \ldots, ((\eta_{j,l} w_{j,p})\circ\alpha_{j,l})(t)\bigr).
\end{align*}
Here, $\eta_{j,l} := (\chi_l \circ\alpha_j)\eta_j \in C^\infty_0(\mathbb{R}^n)$, whereas $\alpha_{j,l}: \mathbb{R}^n \leftrightarrow \mathbb{R}^n$ is an infinitely smooth diffeomorphism such that  $\alpha_{j,l} := \alpha_j^{-1}\circ\alpha_l$ in a neighbourhood of $\mathrm{supp}\, \eta_{j,l}$ and that $\alpha_{j,l}(t) = t$ whenever $|t|\gg1$. Then, given $k\in\{1, \ldots, p\}$, we have the equalities
\begin{align*}
&\Pi_k\bigl(\beta_l\circ(\chi_l w_j)\circ\alpha_l\bigr)(t)\\
= & \Pi_k(\beta_l\circ\beta_j^{-1})\bigl(\alpha_l(t),((\eta_{j,l} w_{j,1})\circ\alpha_{j,l})(t), \ldots, ((\eta_{j,l} w_{j,p})\circ\alpha_{j,l})(t)\bigr) \\
= & \sum_{r=1}^p  \beta_{l,j}^{k,r}(\alpha_l(t))\!\cdot\!((\eta_{j,l} w_{j,r})\circ\alpha_{j,l})(t).
\end{align*}
Here, each $\beta_{l,j}^{k,r}$ is a certain complex-valued function from $C^\infty(\Gamma)$ such that the matrix-valued function $(\beta_{l,j}^{k,r}(x))^p_{k,r=1}$ of $x\in \mathrm{supp}\,\chi_l\cap\mathrm{supp} (\eta_j\circ\alpha_j^{-1})$ corresponds to the transition mapping $\beta_l\circ\beta_j^{-1}$. Thus,
\begin{equation}\label{3.12}
\Pi_k\bigl(\beta_l\circ(\chi_l w_j)\circ\alpha_l\bigr)(t) = \sum_{r=1}^p  \beta_{l,j}^{k,r}(\alpha_l(t))\!\cdot\!((\eta_{j,l} w_{j,r})\circ\alpha_{j,l})(t)
\end{equation}
for arbitrary $t\in \mathbb{R}$ (if $\alpha_l(t)\not\in\Gamma_j$, then this equality becomes $0=0$).

Owing to \eqref{3.8b} and \eqref{3.12} we write
\begin{align*}
\| K\mathbf{w} \|_{s,\varphi;\Gamma, V}^2 & = \sum_{l=1}^\varkappa\sum_{k=1}^p \,\biggl\|\sum_{j=1}^\varkappa \sum_{r=1}^p  (\beta_{l,j}^{k,r}\circ\alpha_l)\!\cdot\!((\eta_{j,l} w_{j,r})\circ\alpha_{j,l})
\biggr\|_{s,\varphi;\mathbb{R}^n}^2\\
& = \sum_{l=1}^\varkappa\sum_{k=1}^p \,\biggl\|\sum_{j=1}^\varkappa \sum_{r=1}^p  \eta_{j,l}^{k,r}\!\cdot\!( w_{j,r}\circ\alpha_{j,l})
\biggr\|_{s,\varphi;\mathbb{R}^n}^2;
\end{align*}
here, each
$$
\eta_{j,l}^{k,r}:=(\beta_{l,j}^{k,r}\circ\alpha_l)\!\cdot\!(\eta_{j,l}\circ\alpha_{j,l})\in C^\infty_0(\mathbb{R}^n).
$$
Thus
\begin{equation}\label{3.13}
\| K\mathbf{w} \|_{s,\varphi;\Gamma, V}^2 \leq \sum_{l=1}^\varkappa\sum_{k=1}^p \biggl(\sum_{j=1}^\varkappa \sum_{r=1}^p\,\|  \eta_{j,l}^{k,r}\!\cdot\!( w_{j,r}\circ\alpha_{j,l})
\|_{s,\varphi;\mathbb{R}^n}\biggr)^2.
\end{equation}

As is known \cite[Theorem B.1.7, B.1.8]{Hermander87}, the operator of change of variables $v \mapsto v\circ\alpha_{j,l}$ and the operator of the multiplication by a function from $C^\infty_0(\mathbb{R}^n)$ are bounded on each Sobolev space $H^{\sigma}(\mathbb{R}^n)$ with $\sigma\in \mathbb{R}$. Therefore the linear operator
$v\mapsto \eta_{j,l}^{k,r}\!\cdot\!(v\circ\alpha_{j,l})$ is bounded on $H^{\sigma}(\mathbb{R}^n)$. Hence, owing to Proposition~\ref{p1}, this operator is also bounded on the H\"ormander space $H^{s,\varphi}(\mathbb{R}^n)$. Thus, formula \eqref{3.13} implies that
\begin{equation}\label{3.16}
\| K\mathbf{w} \|_{s,\varphi;\Gamma, V}^2 \leq c\sum_{j=1}^\varkappa \sum_{r=1}^p \,\|  w_{j,r}\|_{s,\varphi;\mathbb{R}^n}^2
\end{equation}
for a certain number $c>0$ that does not depend on $\mathbf{w}\in (C^\infty_0(\mathbb{R}^n))^{p\varkappa}$. Therefore the mapping \eqref{3.10} extends uniquely (by continuity) to a linear bounded operator
\begin{equation}\label{3.14}
K: (H^{s,\varphi}(\mathbb{R}^n))^{p\varkappa}\to H^{s,\varphi}(\Gamma, V).
\end{equation}
Besides, this mapping extends uniquely to a bounded linear operator
\begin{equation}\label{3.14b}
K: ((H^{\sigma}(\mathbb{R}^n))^{p\varkappa} \rightarrow H^\sigma(\Gamma, V) \quad \mbox{for every}\quad \sigma \in \mathbb{R}^n.
\end{equation}
Taking here $\sigma \in \{ s_0, s_1\}$ and using the interpolation with the function parameter $\psi$, we conclude that the restriction of the operator \eqref{3.14b} with $\sigma = s_0$ to the space \eqref{3.5} is a bounded operator
\begin{equation}\label{3.15}
K: (H^{s,\varphi}(\mathbb{R}^n))^{p\varkappa}  \rightarrow [H^{s-\varepsilon}(\Gamma, V),H^{s+\delta}(\Gamma, V)]_{\psi}.
\end{equation}

It follows from equality \eqref{3.8} and from the boundedness of operators \eqref{3.15} and \eqref{3.3} that
\begin{equation*}
\|u\|_{X_\psi} = \|KTu\|_{X_\psi} \leq c_1 \|u\|_{s,\varphi; \Gamma, V} \quad \mbox{for every}\quad u\in C^\infty(\Gamma, V),
\end{equation*}
where $c_1$ is the norm of the product of these operators, and
\begin{equation*}
X_\psi := [H^{s-\varepsilon}(\Gamma, V),H^{s+\delta}(\Gamma, V)]_{\psi}.
\end{equation*}
Besides, the boundedness of operators \eqref{3.14} and \eqref{3.6} implies that
\begin{equation*}
\|u\|_{s,\varphi; \Gamma, V} = \|KTu\|_{s,\varphi; \Gamma, V} \leq c_2\|u\|_{X_\psi}  \quad \mbox{for every}\quad u\in C^\infty(\Gamma, V),
\end{equation*}
with $c_2$ being the norm of the product of the last two operators. Thus, the norms in the spaces  $H^{s,\varphi} (\Gamma, V)$ and $X_\psi$ are equivalent on the linear manifold $C^\infty(\Gamma, V)$. Since this manifold is dense in these spaces, they are coincide up to equivalence of norms (the set $C^\infty(\Gamma, V)$ is dense in $X_\psi$ due to \eqref{3.18}).
\end{proof}

The proofs of Theorems~\ref{t1}--\ref{t9} are quite similar to the proofs of assertions (i), (iii) and (v) of Theorem 2.3 from monograph \cite[Section 2.1.2]{MikhailetsMurach14}, where the case of trivial vector bundle of rank $p=1$ is considered. We will give these proofs for the sake of the readers convenience and completeness of the presentation.

\begin{proof}[Proof of Theorem $\ref{t1}$]
Consider two triplets $\mathcal{A}_{1}$ and $\mathcal{A}_{2}$ each of which is formed by an atlas of the manifolod $\Gamma$, appropriate partition of unity on $\Gamma$, and collection of local trivializations of the total space $V$. Let $H^{s,\varphi}(\Gamma, V; \mathcal{A}_{j})$ and $H^{\sigma}(\Gamma, V; \mathcal{A}_{j})$ respectively denote the H\"ormander space $H^{s,\varphi}(\Gamma, V)$ and the Sobolev space $H^{\sigma}(\Gamma, V)$ corresponding to the triplet $\mathcal{A}_{j}$ with $j\in\{1,\,2\}$.
The conclusion of Theorem~\ref{t1} holds true in  the Sobolev case of $\varphi\equiv1$ (see, e.g., \cite[p.~110]{Wells}). Hence, the identity mapping is an isomorphism
\begin{equation*}
I:H^{\sigma}(\Gamma, V;\mathcal{A}_{1})\leftrightarrow
H^{\sigma}(\Gamma, V; \mathcal{A}_{2})
\end{equation*}
for each $\sigma\in\mathbb{R}$. Considering this isomorphism for $\sigma := \{s-\varepsilon, s+\delta\}$ and using the interpolation with the function parameter $\psi$ defined by formula \eqref{3.30}, we conclude that the identity  mapping is an isomorphism
\begin{equation*}
I: [H^{s-\varepsilon}(\Gamma, V;\mathcal{A}_{1}), H^{s+\delta}(\Gamma, V;\mathcal{A}_{1})]_\psi \leftrightarrow
[H^{s-\varepsilon}(\Gamma, V;\mathcal{A}_{2}), H^{s+\delta}(\Gamma, V;\mathcal{A}_{2})]_\psi.
\end{equation*}
According to Theorem~\ref{t6},
\begin{equation*}
[H^{s-\varepsilon}(\Gamma, V;\mathcal{A}_{j}), H^{s+\delta}(\Gamma, V;\mathcal{A}_{j})]_\psi = H^{s,\varphi}(\Gamma, V;\mathcal{A}_{j})
\end{equation*}
for each $j\in\{1,\,2\}$ with equivalence of norms in the spaces. Thus, the spaces $H^{s,\varphi}(\Gamma, V;\mathcal{A}_{1})$ and $H^{s,\varphi}(\Gamma, V;\mathcal{A}_{2})$ are equal up to equivalence of norms.
\end{proof}

\begin{proof}[Proof of Theorem~$\ref{t8}$]
Let $\varepsilon>0$. According to Theorem~\ref{t6} there exist interpolation parameters $\chi, \eta \in \mathcal{B}$ such that
\begin{equation*}
[H^{s+\varepsilon/2}(\Gamma, V),H^{s+2\varepsilon}(\Gamma, V)]_{\chi}= H^{s+\varepsilon,\varphi_1}(\Gamma, V),
\end{equation*}
\begin{equation*}
[H^{s-\varepsilon}(\Gamma, V),H^{s+\varepsilon/3}(\Gamma, V)]_{\eta} = H^{s,\varphi}(\Gamma, V),
\end{equation*}
with equivalence of norms. Hence, owing to \eqref{3.18}, we have the continuous embeddings
\begin{equation*}
H^{s+\varepsilon, \varphi_1}(\Gamma, V) \hookrightarrow H^{s+\varepsilon/2}(\Gamma, V) \hookrightarrow H^{s+\varepsilon/3}(\Gamma, V) \hookrightarrow H^{s,\varphi}(\Gamma, V).
\end{equation*}
They are extensions by continuity of the identity mapping $u \mapsto u$, $u\in C^\infty(\Gamma, V)$.
Here, the middle embedding
\begin{equation*}
H^{s+\varepsilon/2}(\Gamma, V) \hookrightarrow H^{s+\varepsilon/3}(\Gamma, V)
\end{equation*}
is compact (see, e.g., \cite[Proposition~1.2]{Wells}). Therefore, the embedding
\begin{equation*}
H^{s+\varepsilon, \varphi_1}(\Gamma, V) \hookrightarrow H^{s,\varphi}(\Gamma, V)
\end{equation*}
is also compact.
\end{proof}

\begin{proof}[Proof of Theorem~$\ref{t9}$]
This theorem is known in the Sobolev case of $\varphi\equiv1$ (see, e.g., \cite[p.~110]{Wells}).
Hence, for every $\sigma\in\mathbb{R}$, the lineal mapping $Q: v \mapsto\langle v, \cdot\rangle_{\Gamma, V}$, with $v\in H^\sigma(\Gamma, V)$,  is an isomorphism $Q: H^{\sigma}(\Gamma, V)\leftrightarrow(H^{-\sigma}(\Gamma, V))'$.
Considering the latter for $\sigma= s \mp 1$ and using the interpolation with the function parameter $\psi$ defined by formula \eqref{3.30} with $\varepsilon=\delta=1$, we obtain  an isomorphism
\begin{equation}\label{7.3}
Q: [H^{s-1} (\Gamma, V), H^{s+1}(\Gamma, V)]_\psi \leftrightarrow [(H^{-s+1} (\Gamma, V))', (H^{-s-1}(\Gamma, V))']_\psi.
\end{equation}
Here,
\begin{equation*}
[H^{s-1} (\Gamma, V), H^{s+1}(\Gamma, V)]_\psi = H^{s, \varphi} (\Gamma, V)
\end{equation*}
by Theorem~\ref{t6}. Besides, according to Proposition~\ref{p5} we have
\begin{align*}
[(H^{-s+1} (\Gamma, V))', (H^{-s-1}(\Gamma, V))']_\psi &= [H^{-s-1} (\Gamma, V), H^{-s+1}(\Gamma, V)]_\chi' \\&= (H^{-s,1/\varphi}(\Gamma, V))'.
\end{align*}
The latter equality is true due to Theorem~\ref{t6} because $\chi(t):=t/\psi(t) = t^{1/2}/\varphi(t^{1/2})$ for $t\geq1$. Thus, isomorphism~\eqref{7.3} acts between the spaces
\begin{equation*}
Q:H^{s,\varphi}(\Gamma,V)\leftrightarrow
(H^{-s,1/\varphi} (\Gamma,V))'.
\end{equation*}
This means that the spaces $H^{s,\varphi} (\Gamma, V)$ and $H^{-s,1/\varphi} (\Gamma, V)$  are mutually dual (up to equivalence of norms) with respect to the sesquilinear form $\langle u,v\rangle_{\Gamma, V}$ of $u\in H^{s, \varphi} (\Gamma, V)$ and $v\in H^{-s, 1/\varphi} (\Gamma, V)$. This form is an extension by continuity of the form $\langle u,v\rangle_{\Gamma, V}$ of $u\in H^{s, \varphi} (\Gamma, V) \hookrightarrow H^{s-1} (\Gamma, V)$ and $v\in H^{-s+1} (\Gamma, V)$. Since the first form is continuous in their arguments separately, estimate~\eqref{4.333} holds true (see, e.g. \cite[Chapter~II, Section~4, Exercise~4]{DunfordSchwartz1958}).
\end{proof}

Our proof of the next Theorem~\ref{t10} is based on the following result.

\begin{proposition}\label{p3}
Let $\varphi\in\mathcal{M}$ and $0\leq q\in\mathbb{Z}$. Then condition \eqref{4.15}
implies the continuous embedding $H^{q+n/2,\varphi}(\mathbb{R}^n)\hookrightarrow C^q_\mathrm{b}(\mathbb{R}^n)$. Conversely, if
\begin{equation}\label{6.1}
\{w\in H^{q+n/2,\varphi}(\mathbb{R}^n): \mathrm{supp}\,w\subset G\} \subset C^q(\mathbb{R}^n)
\end{equation}
for some open nonempty set $G\subset\mathbb{R}^n$, then condition~\eqref{4.15} is satisfied.
\end{proposition}

This proposition follows from H\"ormander's embedding theorem \cite[Theorem 2.2.7]{Hermander63} in the same way as \cite[Theorem~1.15(iii)]{MikhailetsMurach14}.

\begin{proof}[Proof of Theorem $\ref{t10}$]
Let us deduce Theorem~\ref{t10} from Proposition~\ref{p3}.
Suppose first that condition \eqref{4.15} is fulfilled.
Then for an arbitrary section $u\in C^\infty(\Gamma, V)$ we have the inequality
\begin{equation*}
\begin{aligned}
\|u\|_{(q);\Gamma,  V} &= \sum_{j=1}^\varkappa\sum_{k=1}^p\|u_{j,k}\|_{(q);\mathbb{R}^n}
\leq c \sum_{j=1}^\varkappa\sum_{k=1}^p\|u_{j,k}\|_{q+n/2,\varphi;\mathbb{R}^n}\\
&\leq 2^{(p\varkappa  - 1)/2} c \,\biggl(\sum_{j=1}^\varkappa\sum_{k=1}^p\|u_{j,k}\|^2_{q+n/2,\varphi;\mathbb{R}^n}\biggr)^{1/2} = 2^{(p\varkappa  - 1)/2} c \,\|u\|_{q+n/2,\varphi;\Gamma, V}.
\end{aligned}
\end{equation*}
Here, the first equality is due to \eqref{5.25}, and $c$ is the norm of the continuous embedding operator $H^{q+n/2,\varphi}(\mathbb{R}^n)\hookrightarrow C^{q}_{\mathrm{b}}(\mathbb{R}^n)$, which holds due to \eqref{4.15} and Proposition~\ref{p3}.
Hence, the identity mapping $I: u \mapsto u$, with $u\in C^\infty(\Gamma, V)$, extends uniquely (by continuity) to a linear bounded operator
\begin{equation}\label{7.1}
I: H^{q+n/2,\varphi}(\Gamma, V)\rightarrow C^{q}(\Gamma, V).
\end{equation}

If this operator is injective, then it sets the continuous embedding of  $H^{q+n/2,\varphi}(\Gamma, V)$ in $C^{q}(\Gamma, V)$.
Let us prove the injectivity of \eqref{7.1}. Consider the isometric flattening operator \eqref{3.3} with $s = q+n/2$. It is an extension by continuity of mapping \eqref{7.2}. This mapping is well defined on functions $u\in C^{q}(\Gamma, V)$ and sets an isometric operator
$$
T: C^{q}(\Gamma, V)\rightarrow (C^{q}_{\mathrm{b}}(\mathbb{R}^n))^{p\varkappa}.
$$
Therefore the equality $TIu = Tu$ extends by closer from functions $u\in C^\infty(\Gamma, V)$ to functions $u\in H^{q+n/2,\varphi}(\Gamma, V)$. Now, if a section $u\in H^{q+n/2,\varphi}(\Gamma, V)$ satisfies $Iu = 0$, then
$Tu = TIu = 0$ and therefore $u = 0$. Thus, operator \eqref{7.1} is injective, and hence it sets the continuous embedding
\begin{equation}\label{7.4}
H^{q+n/2,\varphi}(\Gamma, V) \hookrightarrow C^{q}(\Gamma, V).
\end{equation}

Let us now proof that this embedding is compact. Without loss of generality we may consider $\varphi\in\mathcal{M}$ as a continuous function on $[1,\infty)$. Indeed, as is knowing \cite[Section~1.4]{Seneta76}, there exists a continuous function  $\varphi_1 \in\mathcal{M}$ that both functions $\varphi/\varphi_1$ and  $\varphi_1/\varphi$ are bounded on $[1,\infty)$. Therefore the spaces $H^{q+n/2,\varphi}(\Gamma, V)$ and $H^{q+n/2,\varphi_1}(\Gamma, V)$ are equal up to equivalence of norms. Then we may use the second space instead of the first in our reasoning.

We put
\begin{equation}
\varphi_0(t) := \varphi(t) \Bigl(\int\limits_t^\infty \frac{dt}{t\varphi(t)}\Bigr)^{1/2}\quad \mbox{for arbitrary}\quad t\geq1.
\end{equation}
Owing to \cite[Lemma~1.4]{MikhailetsMurach14},  the function $\varphi_0$ belongs to $\mathcal{M}$ and has the following two properties:
$$
\lim_{t\rightarrow\infty} \frac{\varphi_0(t)}{\varphi(t)}\rightarrow0\quad\mbox{and}\quad
\int\limits_1^\infty \frac{dt}{t\varphi_0(t)}<\infty.
$$
It follows from the first property that we have the compact embedding
$$
H^{q+n/2,\varphi}(\Gamma, V)\hookrightarrow H^{q+n/2,\varphi_0}(\Gamma, V).
$$
This is demonstrated in the same way as in the proof of \cite[Theorem~2.3(iv)]{MikhailetsMurach14}.
According to the second property, the continuous embedding
$$
H^{q+n/2,\varphi_0}(\Gamma, V)\hookrightarrow C^{q}(\Gamma, V),
$$
holds true, as we have just proved. Hence, embedding~\eqref{7.4} is compact as a composition of compact and continuous embeddings.

It remains to prove that condition~\eqref{4.15} follows from embedding~\eqref{7.4}. Assume that this embedding holds true. Without loss of generality we may suppose that $\Gamma_{1}\not\subset(\Gamma_{2}\cup\cdots\cup\Gamma_{\varkappa})$. Therefore there exists a nonempty open set $U\subset\Gamma_1$ such that $\chi_1(x)=1$ for every $x\in U$. We arbitrarily  choose a function $w\in H^{q+n/2,\varphi}(\mathbb{R}^n)$ such that $\mathrm{supp}\,w \subset \alpha^{-1}_1(U)$. Turn to the operator $K$ defined by formulas \eqref{3.7}--\eqref{3.7b}. According to \eqref{3.14} with $s = q+n/2$ and owing to our assumption, we have the inclusion
\begin{equation}\label{7.7}
K(w, \underbrace{0, \ldots, 0}_{p\varkappa-1})\in H^{q+n/2,\varphi}(\Gamma, V)\subset C^q(\Gamma, V).
\end{equation}
Let us deduce from this inclusion  that $w\in C^q(\mathbb{R}^n)$.

To this end we introduce the linear mapping
$$
T_1: u \mapsto \Pi_1(\beta_1\circ(\chi_1 u)\circ\alpha_1), \quad \mbox{with} \quad u \in C^q(\Gamma, V).
$$
It acts continuously from $C^q(\Gamma, V)$ to $C^q_\mathrm{b}(\mathbb{R}^n)$. Besides, the operator $K$ acts continuously from $(H^{q+n/2,\varphi}(\mathbb{R}^n))^{p\varkappa}$ to $C^q(\Gamma, V)$ according to \eqref{3.14} with $s = q+n/2$ and our assumption.  Hence, $T_1K(w, 0, \ldots, 0) = w$; this equality is evident if additionally $w\in C_0^{\infty}(\mathbb{R}^n)$ and then extends by closure over each function $w$ chosen above. Now
$$
w = T_1K(w, 0, \ldots, 0) \in  C^q(\mathbb{R}^n)
$$
Thus, we obtain embedding~\eqref{6.1} with $G := \alpha^{-1}_1(U)$. It implies condition~\eqref{4.15} due to Proposition~\ref{p3}.
\end{proof}

\section{Proofs of properties of elliptic operators on the refined Sobolev scale}\label{sec8}

Beforehand we will prove the following result:

\begin{lemma}\label{l1}
Let $r\in \mathbb{R}$ and $L \in\Psi^r_{\mathrm{ph}}(\Gamma; V_1, V_2)$. Then the mapping $u\mapsto Lu$, with $u\in C^\infty(\Gamma, V_1)$, extends uniquely (by continuity) to a bounded linear operator
\begin{equation}\label{4.17}
L: H^{\sigma,\varphi}(\Gamma, V_1)\rightarrow H^{\sigma-r,\varphi}(\Gamma, V_2)
\end{equation}
for all $\sigma\in\mathbb{R}$ and  $ \varphi\in \mathcal{M}$.
\end{lemma}

\begin{proof}
Let $\sigma\in \mathbb{R}$ and $\varphi\in\mathcal{M}$. This lemma is known in the Sobolev case of $\varphi=1$ (see, e.g., \cite[p.~92]{Hermander87}).
Thus, the mapping $u\mapsto Lu$, with $u\in C^\infty(\Gamma, V_1)$, extends uniquely to bounded linear operators
$$
L:\,H^{\sigma\mp1}(\Gamma, V_1)\rightarrow H^{\sigma\mp1-r}(\Gamma, V_2).
$$
Using the interpolation with the function
parameter $\psi$ defined by formula \eqref{3.30} with $\varepsilon = \delta = 1$, we conclude by Theorem~\ref{t6} that the restriction of the first operator to the space $H^{s,\varphi}(\Gamma, V_1)$
is a bounded operator between the spaces
\begin{gather*}
L:\,H^{\sigma,\varphi}(\Gamma, V_1)=
\bigl[H^{\sigma-1}(\Gamma, V_1),H^{\sigma+1}(\Gamma, V_1)\bigr]_\psi\rightarrow\\
\rightarrow\bigl[H^{\sigma-r-1}(\Gamma, V_2),H^{\sigma-r+1}(\Gamma, V_2)\bigr]_\psi=
H^{\sigma-r,\varphi}(\Gamma, V_2).
\end{gather*}
\end{proof}

\begin{proof}[Proof of Theorem $\ref{t2}$]
According to Lemma~\ref{l1} the mapping $u\mapsto Au$, with $u\in C^\infty(\Gamma, V_1)$, extends by continuity to  the bounded linear operator~\eqref{4.222}. Let us prove that this operator is Fredholm. Theorem~\ref{t2} is known for Sobolev spaces, where $\varphi = 1$ (see, e.g., \cite[Theorem~19.2.1]{Hermander87}). Therefore the bounded linear operators
\begin{equation}
 A: H^{s\mp1+m}(\Gamma, V_1) \rightarrow H^{s\mp1}(\Gamma, V_2)
\end{equation}
are Fredholm with the kernel $\mathfrak{N}$, index $\dim\mathfrak{N} - \dim\mathfrak{N}^+$, and range
\begin{equation}\label{4.223}
A(H^{s\mp1+m}(\Gamma, V_1)) = \{f\in H^{s\mp1}(\Gamma, V_2): \langle f,w\rangle_{\Gamma, V_2}=0 \;\mbox{for all}\;  w\in \mathfrak{N^+}\}.
\end{equation}
Using the interpolation with the function parameter $\psi$ defined by formula \eqref{3.30} with $\varepsilon = \delta = 1$, we
conclude by Proposition~\ref{p4}  that the bounded operator
$$
A: [H^{s-1+m}(\Gamma, V_1),  H^{s+1+m}(\Gamma, V_1)]_\psi \rightarrow [H^{s-1}(\Gamma, V_2),  H^{s+1}(\Gamma, V_2)]_\psi
$$
is also Fredholm. According to Proposition~\ref{p1} this operator coincides with \eqref{4.222}. Moreover, owing to  Proposition~\ref{p4}, the kernel of the Fredholm operator~\eqref{4.222} equals  $\mathfrak{N}$, the index equals $\dim\mathfrak{N} - \dim\mathfrak{N}^+$, and  the range is
$$
H^{s, \varphi}(\Gamma, V_2) \cap A(H^{s-1+m}(\Gamma, V_1))= \{f\in H^{s,\varphi}(\Gamma, V_2): \langle f,w\rangle_{\Gamma, V_2}=0 \;\mbox{for all}\;  w\in \mathfrak{N^+}\}
$$
in view of \eqref{4.223}.
\end{proof}

\begin{proof}[Proof of Theorem $\ref{t7}$]
Owing to Theorem~\ref{t6}, $\mathfrak{N}$ is the kernel and  $P^+(H^{s,\varphi}(\Gamma, V_2))$ is the range  of the operator~\eqref{4.222}. Hence, the restriction of \eqref{4.222} to the subspace $P(H^{s+m,\varphi}(\Gamma, V_1))$ is the bijective linear bounded operator~\eqref{4.235}. This operator is an isomorphism by the Banach theorem on inverse operator.
\end{proof}

\begin{proof}[Proof of Theorem $\ref{t3}$]
The global estimate \eqref{5.1} is a direct consequence of Theorem~\ref{t2} and Peetre's lemma \cite[p.~728, Lemma~3]{Peetre61}. (Of course, in \eqref{5.1} we may take not only $\sigma<s+m$ but also arbitrary $\sigma\in\mathbb{R}$.) We will deduce $\eqref{f4}$ from $\eqref{5.1}$. Beforehand, let us prove the following result: for each integer $r\geq1$ and for arbitrary functions $\chi,\eta$ from Theorem~\ref{t3} there exists a number $c>0$ such that
\begin{equation}\label{5.6}
\|\chi u\|_{s+m,\varphi;\Gamma, V_1}\leq c\,\bigl(\|\eta Au\|_{s,\varphi;\Gamma, V_2}+\|\eta u\|_{s+m-r,\varphi;\Gamma, V_1}+\|u\|_{\sigma;\Gamma, V_1}\bigr)
\end{equation}
for every $u\in H^{s+m,\varphi}(\Gamma, V_1)$.

According to \eqref{5.1} there exists a number $c_0>0$ such that
\begin{equation}\label{5.3}
\|\chi u\|_{s+m,\varphi;\Gamma, V_1}\leq c_0\,\bigl(\|A(\chi u)\|_{s,\varphi;\Gamma, V_2}+\|\chi u\|_{\sigma;\Gamma, V_1}\bigr)
\end{equation}
for arbitrary $u\in H^{s+m,\varphi}(\Gamma, V_1)$.
Rearranging the PsDO $A$ and the
operator of the multiplication by $\chi$, we arrived at the formula
\begin{equation}\label{5.7}
A(\chi u) = A(\chi\eta u) = \chi A(\eta u) + A'(\eta u) = \chi Au + \chi A((\eta-1) u) + A'(\eta u).
\end{equation}
Here, $A'$ is a certain PsDO from $\Psi_{\mathrm{ph}}^{m-1}(\Gamma; V_1,V_2)$ (see, e.g., \cite[p.~13]{Agranovich94}), and the PsDO $u\mapsto \chi A((\eta-1) u)$ belongs to each class $\Psi_{\mathrm{ph}}^{\lambda}(\Gamma; V_1,V_2)$ with $\lambda\in\mathbb{R}$ because $\mathrm{supp}\, \chi \cap \mathrm{supp}(\eta-1) = \emptyset$. Therefore, owing to Lemma~\ref{l1}, we obtain the inequalities
\begin{equation}\label{5.8}
\begin{aligned}
\|A(\chi u)\|_{s,\varphi;\Gamma,V_2}
&\leq \|\chi Au\|_{s,\varphi;\Gamma,V_2} + \|\chi A((\eta-1) u)\|_{s,\varphi;\Gamma,V_2} + \|A'(\eta u)\|_{s,\varphi;\Gamma,V_2}\\
&\leq \|\chi Au\|_{s,\varphi;\Gamma,V_2} + c_1\|u\|_{\sigma-1,\varphi;\Gamma,V_1} + c_2 \|\eta u\|_{s+m-1,\varphi;\Gamma,V_1}\\
&\leq \|\chi Au\|_{s,\varphi;\Gamma,V_2} + c_1c_3\|u\|_{\sigma;\Gamma,V_1} + c_2 \|\eta u\|_{s+m-1,\varphi;\Gamma,V_1}.
\end{aligned}
\end{equation}
Here, $c_1$ is the norm of the operator $u\mapsto \chi A((\eta-1) u)$ that acts continuously from $H^{\sigma-1,\varphi}(\Gamma, V_1)$ to $H^{s, \varphi}(\Gamma, V_2)$, and $c_2$ is the norm of the bounded operator $A'$ from $H^{s+m-1, \varphi}(\Gamma, V_2)$ to $H^{s,\varphi}(\Gamma, V_1)$. Besides, $c_3$ is the norm of the operator of the continuous embedding $H^{\sigma}(\Gamma, V_1)\hookrightarrow H^{\sigma-1,\varphi}(\Gamma, V_1)$.

Formulas \eqref{5.3} and \eqref{5.8} yield the inequalities
\begin{equation}\label{5.9}
\begin{aligned}
\|\chi u\|_{s+m,\varphi;\Gamma,V_1}&\leq c_0
\bigl(
\|\chi Au\|_{s,\varphi;\Gamma,V_2} + c_1c_3\|u\|_{\sigma;\Gamma,V_1} + c_2 \|\eta u\|_{s+m-1,\varphi;\Gamma,V_1}+\|\chi u\|_{\sigma;\Gamma,V_1}
\bigr)\\
&\leq c_0
\bigl(
\|\chi Au\|_{s,\varphi;\Gamma,V_2} + c_1c_3\|u\|_{\sigma;\Gamma,V_1} + c_2 \|\eta u\|_{s+m-1,\varphi;\Gamma,V_1} + c_4\|u\|_{\sigma;\Gamma,V_1}
\bigr).
\end{aligned}
\end{equation}
Here, $c_4$ is the norm of the bounded operator $u\mapsto\chi u$ on the space $H^{\sigma}(\Gamma,V_1)$. Note that
\begin{equation*}
\|\chi Au\|_{s,\varphi;\Gamma,V_2}=
\|\chi\eta Au\|_{s,\varphi;\Gamma,V_2}\leq
\widetilde{c}\,\|\eta Au\|_{s,\varphi;\Gamma,V_2},
\end{equation*}
with $\widetilde{c}$ being the norm of the bounded operator $v\mapsto\chi v$ on the space $H^{s,\varphi}(\Gamma,V_2)$. Thus, we have proved \eqref{5.6} for $r = 1$.

Choose an integer $k\geq1$ arbitrarily and assume that \eqref{5.6} is true for $r=k$. Let us prove that \eqref{5.6} is also true for $r=k+1$. We choose a function $\eta_1\in C^\infty(\Gamma)$ such that $\eta_1=1$ in a neighbourhood of $\mathrm{supp}\,\chi$ and that $\eta=1$ in a neighbourhood of $\mathrm{supp}\,\eta_1$.  According to our assumption, there exists a number $c_5>0$ such that
\begin{equation}\label{5.10}
\|\chi u\|_{s+m,\varphi;\Gamma, V_1}\leq c_5\,\bigl(\|\eta_1 Au\|_{s,\varphi;\Gamma, V_2}+\|\eta_1 u\|_{s+m-k,\varphi;\Gamma, V_1}+\|u\|_{\sigma;\Gamma, V_1}\bigr)
\end{equation}
for arbitrary $u\in H^{s+m,\varphi}(\Gamma,V_1)$. Owing to \eqref{5.1} we write
\begin{equation}\label{5.11}
\|\eta_1 u\|_{s+m-k,\varphi;\Gamma,V_1}\leq c_6
\bigl(
\|A(\eta_1 u)\|_{s-k,\varphi;\Gamma,V_2} + \|\eta_1 u\|_{\sigma;\Gamma,V_1}
\bigr);
\end{equation}
here, $c_6$ is a certain positive number that does not depend on $u$.
Rearranging the PsDO $A$ and the
operator of the multiplication by $\eta_1$, we obtain
\begin{equation}\label{5.12}
A(\eta_1 u) = A(\eta_1\eta u) = \eta_1 A(\eta u) + A'_1(\eta u) = \eta_1 Au + \eta_1 A((\eta-1) u) + A'_1(\eta u).
\end{equation}
Here, $A'_1$ is a certain PsDO from $\Psi_{\mathrm{ph}}^{m-1}(\Gamma; V_1,V_2)$, and the PsDO $u\mapsto\eta_1 A((\eta-1) u)$ belongs to each class $\Psi_{\mathrm{ph}}^{\lambda}(\Gamma; V_1,V_2)$ with $\lambda\in\mathbb{R}$ because $\mathrm{supp}\, \eta_1 \cap \mathrm{supp} (\eta-1) = \emptyset$. Therefore, owing to Lemma~\ref{l1}, we obtain the inequalities
\begin{equation}\label{5.13}
\begin{aligned}
\|A(\eta_1 u)\|_{s-k,\varphi;\Gamma,V_2}
&\leq \|\eta_1 Au\|_{s-k,\varphi;\Gamma,V_2} + \|\eta_1 A((\eta-1) u)\|_{s-k,\varphi;\Gamma,V_2} +
\|A'_1(\eta u)\|_{s-k,\varphi;\Gamma,V_2}\\
&\leq \|\eta_1 Au\|_{s-k,\varphi;\Gamma,V_2} + c_7\|u\|_{\sigma-1,\varphi;\Gamma,V_1} +
c_8\|\eta u\|_{s-k+m-1,\varphi;\Gamma,V_1}\\
&\leq \|\eta_1 Au\|_{s-k,\varphi;\Gamma,V_2} + c_7c_3\|u\|_{\sigma;\Gamma,V_1} +
c_8\|\eta u\|_{s+m-(k+1),\varphi;\Gamma,V_1}.
\end{aligned}
\end{equation}
Here, $c_7$ is the norm of the operator $u\mapsto \eta_1 A((\eta-1) u)$ that acts continuously from $H^{\sigma-1,\varphi}(\Gamma, V_1)$ to $H^{s-k, \varphi}(\Gamma, V_2)$, and $c_8$ is the norm of the bounded operator $A'_1$ from $H^{s-k+m-1, \varphi}(\Gamma, V_2)$ to $H^{s-k,\varphi}(\Gamma, V_1)$.

Now formulas \eqref{5.10}, \eqref{5.11}, and \eqref{5.13} yield the inequalities
\begin{equation}\label{5.17}
\begin{aligned}
\|\chi u\|_{s+m,\varphi;\Gamma, V_1}
&\leq c_5\,\Bigl(\|\eta_1 Au\|_{s,\varphi;\Gamma, V_2}+c_6
\bigl(
\|A(\eta_1 u)\|_{s-k,\varphi;\Gamma,V_2} + \|\eta_1 u\|_{\sigma;\Gamma,V_1}
\bigr) + \|u\|_{\sigma;\Gamma, V_1}\Bigr)
\\&\leq c_5
\Bigl(
\|\eta_1 Au\|_{s,\varphi;\Gamma,V_2} +
c_6\bigl(\|\eta_1 Au\|_{s-k,\varphi;\Gamma,V_2} + c_7c_3\|u\|_{\sigma;\Gamma,V_1}
\\&\quad\quad\quad\quad\quad\quad\quad\quad\quad+ c_8\|\eta u\|_{s+m-(k+1),\varphi;\Gamma,V_1}
\bigr)
+ \|u\|_{\sigma;\Gamma,V_1}
\Bigr).
\end{aligned}
\end{equation}
Since $\eta_1=\eta_1\eta$, we have
\begin{equation}\label{5.18}
\begin{aligned}
\|\eta_1 Au\|_{s,\varphi;\Gamma,V_2} + c_6\|\eta_1 Au\|_{s-k,\varphi;\Gamma,V_2} &\leq
(1+c_6)\|\eta_1 \eta Au\|_{s,\varphi;\Gamma,V_2} \\&\leq
(1+c_6)c_{9}\|\eta Au\|_{s,\varphi;\Gamma,V_2}.
\end{aligned}
\end{equation}
Here, $c_9$ is the norm of the bounded operator $v\mapsto \eta_1 v$ on the space $H^{s,\varphi}(\Gamma, V_2)$.
Now formulas \eqref{5.17} and \eqref{5.18} give the inequality \eqref{5.6} with $r=k+1$.
Owing to the principle of mathematical induction, this inequality is true for each integer $r\geq1$.

The required estimate \eqref{f4} follows from the inequality \eqref{5.6}, where $r\in\mathbb{Z}$ such that $s+m-r<\sigma$, in  view of
\begin{align*}
\|\eta u\|_{s+m-r,\varphi;\Gamma, V_1}\leq
c_{10}\|\eta u\|_{\sigma;\Gamma, V_1}\leq
c_{10}c_{11}\|u\|_{\sigma;\Gamma, V_1}.
\end{align*}
Here, $c_{10}$ is the norm of the embedding operator $H^{s+m-r,\varphi}(\Gamma, V_1) \hookrightarrow H^{\sigma}(\Gamma, V_1)$, and $c_{11}$ is the norm of the operator $u \mapsto \eta u$ on the space $H^{\sigma}(\Gamma, V_1)$
\end{proof}

As to Remark~\ref{r1} note that inequality \eqref{5.2} follows from \eqref{5.9} with $\sigma<s+m-1$ in view of Theorem~\ref{t8}.

\begin{proof}[Proof of Theorem $\ref{t5}$]
Since $u\in H^{-\infty}(\Gamma, V_1)$, there exists an integer $r\geq0$ such that
$u\in H^{s+m-r,\varphi}(\Gamma, V_1)$. Let us first prove this theorem in the global case where $\Gamma_0=\Gamma$. In this case, $$
Au=f\in H^{s,\varphi}(\Gamma, V_2)\cap
A(H^{s+m-r,\varphi}(\Gamma, V_1))=
A(H^{s+m,\varphi}(\Gamma, V_1))
$$
by the condition and Theorem~\ref{t2}.
Hence,  there exists a section $v\in H^{s+m,\varphi}(\Gamma, V_1)$ such that $Av=f$ on $\Gamma$.
Since $A(u-v)=0$ on $\Gamma$ and $u-v \in H^{s+m-r,\varphi}(\Gamma, V_1)$, we conclude by Theorem~\ref{t6} that
$$
w:=u-v \in \mathfrak{N}\subset C^\infty(\Gamma, V_1).
$$
Thus,
$$
u=v+w\in H^{s+m,\varphi}(\Gamma, V_1).
$$
Theorem~\ref{t5} is proved in the case of $\Gamma_0=\Gamma$.

We now deduce this theorem in the general situation from the case just considered.
Beforehand, let us prove that for every integer $k\geq 1$ the following implication holds for $u$:
\begin{equation}\label{4.13}
u\in H_{\mathrm{loc}}^{s+m-k,\varphi}(\Gamma_0, V_1)\; \Rightarrow \;
u \in H_{\mathrm{loc}}^{s+m-k+1,\varphi}(\Gamma_0, V_1).
\end{equation}
Assume that $u\in H_{\mathrm{loc}}^{s+m-k,\varphi}(\Gamma_0, V_1)$. We  arbitrarily choose a function $\chi\in C^{\infty}(\Gamma)$ such that  $\mathrm{supp}\,\chi\subset \Gamma_0$. Let a function $\eta\in C^{\infty}(\Gamma)$ satisfy the conditions
$\mathrm{supp}\,\eta\subset \Gamma_0$ and $\eta=1$ in a neighbourhood of $\mathrm{supp}\,\chi$. According to \eqref{5.7} we have the equality
\begin{equation}\label{5.21}
A(\chi u) = \chi f + \chi A((\eta-1) u) + A'(\eta u).
\end{equation}
Here, $\chi f \in H^{s,\varphi}(\Gamma, V_2)$ by the condition; $\chi A((\eta-1) u)\in H^{s,\varphi}(\Gamma, V_2)$ because $u\in H^{s+m-r,\varphi}(\Gamma, V_1)$ and the PsDO $u\rightarrow \chi A((\eta-1) u)$ belongs to $\Psi_{\mathrm{ph}}^{m-r}(\Gamma; V_1,V_2)$, and $A'(\eta u)\in H^{s-k+1,\varphi}(\Gamma, V_2)$ because $\eta u \in H^{s+m-k,\varphi}(\Gamma,V_1)$ by our assumption and because the inclusion $A'\in\Psi_{\mathrm{ph}}^{m-1}(\Gamma; V_1,V_2)$. Hence, the right-hand side of equality \eqref{5.21} belongs to $H^{s-k+1,\varphi}(\Gamma, V_2)$. Therefore $\chi u\in H^{s+m-k+1,\varphi}(\Gamma, V_1)$ by what we have proved in the previous paragraph. Thus,  $u \in H_{\mathrm{loc}}^{s+m-k+1,\varphi}(\Gamma_0, V_1)$ in view of our choice of $\chi$. Implication \eqref{4.13} is proved.

Applying this implication successively for $k=r,\,r-1,\ldots,1$, we conclude that
\begin{gather*}
u\in H^{s+m-r,\varphi}(\Gamma, V_1)\subset
H_{\mathrm{loc}}^{s+m-r,\varphi}(\Gamma_0, V_1)\Rightarrow\\
\Rightarrow u \in H_{\mathrm{loc}}^{s+m-r+1,\varphi}(\Gamma_0, V_1)
\Rightarrow \ldots\Rightarrow u \in
H_{\mathrm{loc}}^{s+m,\varphi}(\Gamma_0, V_1).
\end{gather*}
Thus, we have proved the required inclusion $u\in H_{\mathrm{loc}}^{s+m,\varphi}(\Gamma_0, V_1)$.
\end{proof}

\begin{proof}[Proof of Theorem $\ref{t4}$]
Owing to Theorem~\ref{t5} where $s:=q-m+n/2$ we have the inclusion $u\in H_{\mathrm{loc}}^{q+n/2,\varphi}(\Gamma_0, V_1)$. We arbitrarily choose a function $\chi\in C^{\infty}(\Gamma)$ such that $\mathrm{supp}\,\chi\subset \Gamma_0$. Then
$$
\chi u \in H^{q+n/2,\varphi}(\Gamma, V_1) \hookrightarrow C^q(\Gamma, V_1)
$$
due to condition \eqref{4.15} and Theorem~\ref{t10}. Therefore $u \in C^q(\Gamma_0, V_1)$.
\end{proof}

\begin{proof}[Proof of Remark $\ref{r2}$]
If condition~\eqref{4.15} is satisfied, then we have implication~\eqref{4.16} according to Theorem~\ref{t4}.
Assume now that this implication is valid and prove that $\varphi$ satisfies condition~\eqref{4.15}.
Without loss of generality we may suppose that $\Gamma_0\cap\Gamma_1\neq\emptyset$.  We choose a nonempty open set $U\subset\Gamma_0\cap\Gamma_1$ and a function $\chi\in C^{\infty}(\Gamma)$ such that  $\mathrm{supp}\,\chi \subset \Gamma_0$ and $\chi=1$ on $U$. Turn to the operator $K$ defined by formulas \eqref{3.7}--\eqref{3.7b}, where $\beta_j := \beta_{1,j}$ and the function $\eta_1$ additionally satisfies the equality $\eta_1= 1$ on the set $\alpha_{1}^{-1}(U)$.
We arbitrarily  choose a function $w\in H^{q+n/2,\varphi}(\mathbb{R}^n)$ such that $\mathrm{supp}\,w \subset \alpha^{-1}_1(U)$.
Then
\begin{equation*}
u := K(w, \underbrace{0, \ldots, 0}_{p\varkappa-1})\in H^{q+n/2,\varphi}(\Gamma, V_1)
\end{equation*}
according to \eqref{3.14} with $s := q+n/2$ and $V:=V_1$. The premise of implication~\eqref{4.16} holds true for the section $u$. Hence, $u\in C^q_{\mathrm{loc}}(\Gamma_0, V_1)$ according to this implication. Therefore $u = \chi u\in C^q(\Gamma, V_1)$ due to our choice of $\chi$.
Let us use the operator $T_1$, with $\beta_1 := \beta_{1,1}$, introduced in the proof of Theorem~\ref{t10}. Owing to the properties of $T_1$ mentioned therein, we write
$$
w = T_1K(w, 0, \ldots, 0) \in  C^q(\mathbb{R}^n).
$$
Thus, we obtain embedding~\eqref{6.1} with $G := \alpha^{-1}_1(U)$. It implies condition~\eqref{4.15} by Proposition~\ref{p3}.
\end{proof}

\textit{The author is grateful to A. A. Murach for his big help in preparing of the paper. The author thanks Referees for their remarks and suggestions about improving the language of the paper.}

\end{document}